\documentclass [12pt] {amsart}
\usepackage [latin1]{inputenc}
\usepackage[all]{xy}
\usepackage[english]{babel} 
\usepackage{amssymb}
\usepackage[lite, initials]{amsrefs}
\newtheorem{teorema}{Theorem}[section]
\newtheorem*{theorem*}{Structure Theorem}
\newtheorem{lemma}[teorema]{Lemma}
\newtheorem{propos}[teorema]{Proposition}
\newtheorem{corol}[teorema]{Corollary}
\theoremstyle{definition}
\newtheorem{ex}{Example}[section]
\newtheorem{rem}{Remark}[section]
\newtheorem{defin}[teorema]{Definition}

\def\R{{\mathbb R}}  
\def\N{{\mathbb N}}  
\def\C{{\mathbb C}}  
\def\Q{{\mathbb Q}}

\def\D{\Delta}

\def\oli{\overline}

\newcommand{\Chi}{\mathfrak{X}}

\newcommand{\hcal}{\mathcal H}
\def\Int{\stackrel{\rm o}}
\def\rdim{{\rm dim}_{\R}}
\def\rcodim{{\rm codim}_{\R}}

\def\oli{\overline}				
\def\uli{\underline}   
\def\p{\partial}

\def\Crt{{\rm Crt}}

\newcommand{\Min}{{\rm Min}}

\title[WEAKLY COMPLETE COMPLEX SURFACES]{WEAKLY COMPLETE COMPLEX SURFACES}
\author[S.~Mongodi]{Samuele Mongodi\textsuperscript{1}}
\address[\textsuperscript{1}]{Dipartimento di Matematica - Universit\`a di ``Tor Vergata'', Via della Ricerca Scientifica --- I-00133 Roma, Italy}
\email{mongodi@mat.uniroma2.it} 
\author[Z.~Slodkowski]{Zbigniew Slodkowski\textsuperscript{2}}
\address[\textsuperscript{2}]{Department of Mathematics, University of Illinois at Chicago, 851 South Morgan Street, Chicago, Illinois 60607, Usa}
\email{zbigniew@uic.edu}
\author[G.~Tomassini]{Giuseppe Tomassini\textsuperscript{3}}
\address[\textsuperscript{3}]{Scuola Normale Superiore, Piazza dei Cavalieri, 7 - I-56126 Pisa, Italy}
\email{g.tomassini@sns.it}

  \date{\today}

\begin{document}
\begin{abstract}
A weakly complete space is a complex space admitting a (smooth) plurisubharmonic exhaustion function. In this paper, we classify those weakly complete complex surfaces for which such exhaustion function can be chosen real analytic: they can be modifications of Stein spaces or proper over a non compact (possibly singular) complex curve or foliated with real analytic Levi-flat hypersurfaces which in turn are foliated by dense complex leaves (these we call surfaces of Grauert type). In the last case, we also show that such Levi-flat hypersurfaces are in fact level sets of a global proper pluriharmonic function, up to passing to a holomorphic double cover of the space.

An example of Brunella shows that not every weakly complete surface can be endowed with a real analytic plurisubharmonic exhaustion function.

Our method of proof is based on the careful analysis of the level sets of the given exhaustion function and their intersections with the \emph{minimal singular set}, i.e the set where every plurisubharmonic exhaustion function has a degenerate Levi form.\end{abstract}
\maketitle
\tableofcontents
\section{Introduction}\label{intr}\noindent
A {\em weakly complete} complex space is a connected complex space $X$ endowed with a
smooth plurisubharmonic exhaustion function $\varphi$. 

Weakly complete subdomains of Stein spaces are Stein; complex spaces which are proper over (i.e. endowed with a proper surjective holomorphic  map onto) a Stein space, e.g. modifications of Stein spaces, are weakly complete.

More interesting examples of weakly complete complex spaces which are not Stein are the pseudoconvex subdomains of a complex torus constructed by Grauert (cfr. \cite{Nar}). 

Let us briefly compare these two classes of examples. In the former, $\mathcal{O}(X)\neq\C$ and, if there is a positive dimensional complex space contained in a level set of $\varphi$, then it is a compact subspace of $X$. Such compact subspaces are responsible for the degeneration of the Levi form of $\varphi$ (or of any other plurisubharmonic function). In the latter, $\mathcal{O}(X)=\C$ and there exists a smooth plurisubharmonic exhaustion function $\varphi$ whose regular level sets are Levi-flat hypersurfaces, foliated by dense complex leaves (along which the Levi form of $\varphi$ degenerates). In such a situation $X$ is said to be a surface of \emph{Grauert type}.

A question naturally arises: are these two phenomena the only possible for a weakly complete space?

Even if this problem has never been explicitly addressed, we find some partial results throughout the literature about the various forms of the Levi problem. For instance, in \cite{Ohs}, Ohsawa shows (Proposition 1.4) that a weakly complete surface (i.e. 2-dimensional complex manifold) with a non-constant holomorphic function is holomorphically convex, hence proper on a Stein space, by Remmert's theorem. Later on, Diederich and Ohsawa tackle a weak form of the Levi problem for a domain with real analytic boundary (cfr. \cite{DieOhs}). More recently, in \cite{GMO} Gilligan, Miebach and Oeljeklaus study the case of a pseudoconvex domain of any dimension, spread over a complex homogeneous manifold.

The general problem is already hard for complex surfaces thus, before studying in complete detail the general weakly complete surface $X$, it seems worthwile to analyse the special case of surfaces admitting a \emph{real analytic} plurisubharmonic exhaustion function. In this vein Brunella (cfr.{\cite{Bru}) constructed an example of weakly complete surface which does not admit a real analytic plurisubharmonic exhaustion function. In order to do so, he studied the possible geometries of some special weakly complete surfaces, namely, those obtained as the complement of a curve with topologically trivial normal bundle in a compact (algebraic) surface not containing any $(-2)$-curve (i.e. smooth, rational, with self-intersection $-2$). He showed that, in presence of a real analytic plurisubharmonic exhaustion function, these particular instances of weakly complete surfaces are either Stein, proper over a Stein space, or of Grauert type.

In this paper, we prove that this classification holds in general. Namely, we obtain the following
\begin{theorem*}Let $X$ be a weakly complete complex surface, with a real analytic plurisubharmonic exhaustion function $\alpha$. Then one of the following three cases occurs:
\begin{itemize}
\item[i)] $X$ is a modification of a Stein space of dimension $2$,
\item[ii)] $X$ is proper over a (possibly singular) open complex curve,
\item[iii)]$X$ is a Grauert type surface.
\end{itemize}
Moreover, in case \emph{iii)}, either the critical set $\Crt(\alpha)$ of $\alpha$ has dimension $\le 2$ and then 
\begin{itemize}
\item[iii-a)] the absolute minimum set $Z$ of $\alpha$ is a compact complex curve $Z\subset X$ and there exists a proper pluriharmonic function $\chi:X\setminus Z\to \R$ such that every plurisubharmonic function on $X\setminus Z$ is of the form $\lambda\circ\chi$,
\end{itemize}
or it is of has dimension $3$ and then
\begin{itemize}
\item[iii-b)] there exist a double holomorphic covering map $\pi:X^*\to X$ and a proper pluriharmonic function $\chi^*:X^*\to \R$ such that every plurisubharmonic function on $X^*$ is of the form $\lambda\circ\chi^*$.
\end{itemize}
\end{theorem*}

\noindent The first part of the Structure Theorem is the content of Theorem \ref{MAR1}, the second one of the theorems \ref{SCOTTISH} and \ref{maint}.
  
\medskip
  
  The method we adopted to tackle the problem consists of a careful analysis of the structure of the level sets of $\alpha$ and their behaviour with respect to the \emph{minimal singular set}. This set was introduced in \cite{ST} for any weakly complete complex space; let us recall its definition. Given any plurisubharmonic exhaustion function $\varphi$ in $X$, let $\Sigma^1_\varphi$ be the minimal closed set such that $\varphi$ is strictly plurisubharmonic on $X\setminus\Sigma^1_\varphi$, and set $\Sigma^1=\Sigma^1(X)=\bigcap\limits_{\varphi}\Sigma^1_\varphi$, i.e.  $x\in\Sigma^1$ if no
plurisubharmonic exhaustion function is strictly plurisubharmonic near
$x$.
  
A plurisubharmonic exhaustion function $\varphi$ is called \emph{minimal} if $\Sigma^1=\Sigma^1_\varphi$. 

The following crucial properties were proved in \cite{ST} when $X$ is a complex manifold:
\begin{itemize}
\item[a)] there exist minimal functions $\varphi$ (cfr. \cite[Lemma 3.1]{ST});
\item[b)] if $\varphi$ is minimal the nonempty level sets 
$\Sigma^1_c=\{\varphi=c\}\cap\Sigma^1$
have the local maximum property (cfr. \cite[Theorem 3.6]{ST});
\item[c)] if ${\rm dim}_{_\C}X=2$ and $c$ is a regular value of
$\varphi$, then the (nonempty) level sets $\Sigma^1_c$ are compact sets foliated
by Riemann surfaces (cfr. \cite[Lemma 4.1]{ST}).
\end{itemize}
If we restrict ourselves to the real analytic category, while property a) still holds, the proofs for b) and c) given in \cite{ST} do not apply anymore. However, if $X$ is a weakly complete complex surface, it is possible to show that b) and c) hold for any plurisubharmonic exhaustion function and not just for the minimal ones (see Theorem \eqref{lmp}).

Thus we are able to link $\Sigma^1$ to the Levi-flatness of the levels of $\alpha$, obtaining the first half of the Structure Theorem, by studying the complex foliation induced by the degeneracy of the Levi form.

Then, we proceed to construct, for Grauert type surfaces, a proper pluriharmonic function.

The paper consists of six sections and an appendix.

In Section 2, we present some relevant examples of the more ``exotic'' of the three cases, the Grauert type surfaces. Section 3 is devoted to the study of the regular level sets of $\alpha$ which intersect $\Sigma^1$. We show that they are Levi-flat, provided that they do not contain any compact complex curve, and that the Levi foliation has dense leaves.

The effects of the presence of compact curves is analyzed in Section 4, where, using in a crucial way a theorem of Nishino \cite[III.5.B]{Ni}, we prove the Theorem \ref{MAR1}, i.e. the first part of the Structure Theorem, giving the classification of weakly complete surfaces into cases i), ii) and iii).

The proof of parts iii-a), iii-b) of the Structure Theorem is given in Sections 5,6 where we are dealing with Grauert type surfaces. Our goal is to construct a proper pluriharmonic function in the Structure Theorem. This is done analyzing the cases $\rdim\Crt(\alpha)\le 2$ and $\rdim\Crt(\alpha)=3$ separately (see Theorem \eqref{SCOTTISH} and  Theorem \ref{maint}).

The Appendix collects results of various nature which we believe either known or easy to proof, but for which we were not able to find appropriate references.

\medskip

We would like to thank Tadeusz Mostowski for directing us to the Sullivan's paper. 

\medskip

{\bf Acknoledgements.} The first author was supported by the ERC grant HEVO -- Holomorphic Evolution Equations n. 277691.

\section{Examples}\label{exm}
\begin{ex}
Let $a_1,a_2\in\C$ with the following properties
$$
0<\vert a_1\vert\le\vert a_2\vert<1 ,\>\>a_1^k\neq a_2^l 
$$
for all $(k,l)\in\N^2\smallsetminus\{(0,0)\}$ and define $\tau$ by $\vert a_1\vert=\vert a_2\vert^\tau$; by hypothesis $\tau\notin\Q$.

Consider on $\C^2\smallsetminus\{(0,0)\}$ the equivalence relation $\sim$: $(z_1,z_2)\sim (a_1z_1,a_2z_2)$. The quotient space $\C^2\smallsetminus\{(0,0)\}/\sim$ is the Hopf manifold $\mathcal H$. Let $\pi$ denote the projection $\C^2\smallsetminus\{(0,0)\}\rightarrow\mathcal H$. The complex lines $\C_{z_1}=\{z_2=0\}$, $\C_{z_2}=\{z_1=0\}$ project into complex compact curves $C_1$, $C_2$ respectively. Consider $X=\hcal\smallsetminus C_2$. The function
$$
\Phi(z_1,z_2)=\frac{\vert z_2\vert^{2\tau}}{\vert z_1\vert^2}
$$
on $\C^2\smallsetminus\{(0,0)\}$ is $\sim$-invariant and so defines a function $\phi:X\rightarrow\R_{\{\ge 0\}}$; $\phi$ is proper and $\log\,\phi$ is pluriharmonic on $X\smallsetminus C_1$.The level sets of $\phi$ contained in $X\smallsetminus C_1$ are the projections of the sets $\vert z_1\vert=c\vert z_2\vert^\tau$, $c>0$, and so foliated by the projections of the sets $ z_1=c{\rm e}^{i\theta} z_2^\tau$ which are everywhere dense leaves, $\tau$ being irrational. Observe that $C_1$ is the minimum set of $\phi$. In particular, $X$ is a weakly complete surface of Grauert type and falls into case iii-a.
\end{ex}
\begin{ex}
With the notation of the previous example, we consider $X_1=\mathcal{H}\setminus(C_1\cup C_2)$ with plurisubharmonic exhaustion function $\alpha=(\log\phi)^2$. $X_1$ is a weakly complete surface, obviously of Grauert type. Here, however, the plurisubharmonic function $\alpha_1$ has a 3-dimensional minimum set, namely the quotient of the Levi-flat surface of $(\C^*)^2$ given by
$$H_0=\{(z_1,z_2)\in(\C^*)^2\ :\ |z_2|^\tau=|z_1|\}\;.$$
The pluriharmonic function on $X_1$ is, obviously, $\log(\phi)$, i.e. a befitting choice of the square root of $\alpha_1$. Therefore, $X$ falls into case iii-b.
\end{ex}
\begin{ex}
Another class of example is provided by total spaces of some complex line bundles over compact Riemann
surfaces (see also \cite{Ued}).

Let $M$ be a compact Riemann surface of genus $g>0$. It is well known that every topologically trivial line bundle can be represented by a flat unimodular cocycle, i.e. an element of $H^1(M, S^1)$.

Consider a line bundle $L\to M$ with trivialization given by the open covering $\{U_j\}_{j=1}^n$ and transition functions $\{\xi_{ij}\}_{i,j}$ which represent a cocycle $\xi\in H^1(M,S^1)$. We can define a function $\alpha:L\to \R$ by defining it on each trivialization as $\alpha_j:U_j\times\C$, $\alpha_j(x,w)=|w|^2$. As $|\xi_{ij}|=1$, these functions glue into $\alpha:L\to\R$, which is readily seen to be plurisubharmonic and exhaustive.

Now, consider $r>0$ and the section $f_1\in\Gamma(U_1, \xi)$ given by $f(x)\equiv r$ for all $x\in U_1$; taking all possible analytic continuations of $f_1$ as a section of the bundle $L$, we construct , for every chain  $\{U_{j_k}\}_{k\in\N}$ with $j_0=1$ and $U_{j_{k}}\cap U_{j_{k+1}}\neq\varnothing$, the sections
$f_k\equiv\xi_{j_{k}j_{k-1}}\xi_{j_{k-1}j_{k-2}}\cdots\xi_{j_1j_0}r\in\Gamma(U_{j_k}, \xi)\;.$
Representing $\xi$ as a multiplicative homomorphism $\psi_\xi:\pi_1(M)\to S^1$, it is easy to see that the graphs of such sections glue into a compact complex manifold if and only if $\psi(\pi_1(M))$ is contained in the roots of unity, i.e. if and only if $L^{\otimes n}$ is (analytically) trivial for some $n$.

If that is not the case, the graphs of such sections glue into an imbedded, non closed, complex manifold, contained in  the Levi-flat hypersurface $\alpha^{-1}(r^2)$ and dense in it. The other leaves of the Levi foliation are obtained by the one constructed multiplying it by $e^{i\theta}$.

Finally, we have a pluriharmonic function $\chi:L\setminus M\to\R$ given by $\chi(p)=\log\alpha(p)$. 

Therefore, if $\xi$ is not unipotent, its total space gives an example of Grauert type surface for the case iii-a. As with the Hopf surface example, it is not hard to show that $L\setminus M$ is a Grauert type surface and falls in case iii-b.
\end{ex}
\section{Levels and defining functions}\label{levels}
\subsection{Some preliminary remarks}
From now on we suppose that $X$ is a fixed complex surface that admits a real analytic plurisubharmonic exhaustion function $\alpha$.
\begin{propos}\label{sorem}
We can assume $\alpha$ to have the following properties
\begin{equation}\label{Osric}
\begin{cases}
\min\limits_{X}\,\alpha=0&\\
\p\alpha(p)=0\>\>{\rm if}\>\>\p\bar\p\alpha(p)=0.&
\end{cases}\end{equation}
\end{propos}
\begin{proof} It is enough to replace $\alpha$ with $(\alpha-\min\limits_X\,\alpha)^2$.
 \end{proof}

Let us fix some notations. Given a smooth function $f:W\to\R$ on a complex surface $W$ let us denote ${\rm Crt}(f)$ 
the set of its critical points. 
 If $f$ is real anlytic ${\rm Crt}(\alpha)$ is a real analytic set.

As anticipated in Section \ref{intr}, property b) of minimal functions extends to arbitrary smooth plurisubharmonic exhaustion functions: 
  \begin{teorema}\label{lmp}
 Let $\Sigma^1$ be the minimal kernel of $X$ and $u:X\to\R$ an arbitrary smooth plurisubharmonic exhaustion function. Then every non-empty
$$
\Sigma^1_u(c)=\{u=c\}\cap\Sigma^1
$$
has the local maximum property. 
\end{teorema}
We first prove the following lemmas
\begin{lemma}\label{lmp1}
Let $Z$ be a complex manifold, $Y_1$ a closed subset with the local maximum property and $Y_0$ a closed subset of $Y_1$. Assume that $\phi:W\to[-\infty,+\infty)$ is a plurisubharmonic function $\phi$ in (a neighbourhood of) $Y_1$ such that $\phi=c$ on $Y_0$ and $\phi<c$ on $Y_1\setminus Y_0$. Then $Y_0$ has the local maximum property. 
\end{lemma}
\begin{proof}
Should the local maximum property for $Y_0$ fail then, by \cite[Proposition 2.3]{Sl2} there are $y\in Y_0$ a coordinate ball $B=B(y;r)$ and a strongly plurisubharmonic function $u$ on $\oli B$ such that: $u(y)=0$ and $u(z)<-\epsilon\vert z-y\vert^2$, , if $z\in B\cap Y_0\setminus\{y\}$. In particular, there exist a neighbourhood $V$ of ${\rm b}B$ and a positive number $\epsilon$ such that $u(z)<-\epsilon\vert z-y_0\vert^2$ for $z\in B\cap Y_1\cap {\rm b}B$. Then for $m\gg0$ the plurisubharmonic function $v:=u+m(\phi-c)$ satisfies $v(y)=0$, $v(z)<-\epsilon\vert z-y\vert^2$ for ${\rm b}B\cap Y_1$.  
Therefore, since $Y_1$ has the local maximum property, this inequality must hold on $B\cap Y_1$, a contradiction for $z=y$.
 \end{proof}
\begin{lemma}\label{lmp2}
If $\psi:X\to\R$ is a smooth plurisubharmonic exhaustion function, $\phi:X\to\R$ a minimal function and $c$, $d$ real numbers such that the set
$$
Y_0=\big\{x\in X:x\in\Sigma^1, \psi(x)=c,\phi(x)=d\big\}
$$
is non-empty. Then $Y_0$ has the local maximum property.
\end{lemma}
\begin{proof}
Assume that $v:\R^2\to\R$ is a smooth, strictly convex function such that for some numbers $K$, $L$ the functions $v_x$, $v_y$ are positive on the set
$$
\big\{(x,y)\in \R^2:x\ge K, y\ge L\big\}
$$
and that $\min\psi>K$, $\min\phi>L$. Let $\chi(z)=v\big(\psi(z),\phi(z)\big)$, $z\in X$. Then $\chi$ is a minimal function. Indeed
$$
{\rm d}{\rm d^c}\chi=\left[ \begin{array}{cc}
	v_{xx}&v_{xy}\\
	v_{yx}&v_{yy}
\end{array} \right]\left[ \begin{array}{c} {\rm d}\phi\\{\rm d}\psi\end{array} \right]\wedge\left[ \begin{array}{c} {\rm d^c}\phi\\{\rm d^c}\psi\end{array} \right]+v_x{\rm d}{\rm d^c}\phi+v_y{\rm d}{\rm d^c}\psi
$$
Then ${\rm d}{\rm d^c}\chi$ is strictly positive definite at all points where ${\rm d}{\rm d^c}\phi$ is and so $\chi$ is minimal.

If $K$, $L$ are fixed numbers such that $\min\psi>K$, $\min\phi>L$ we can take $v(x,y)=(x-L)^2+(y-K)^2$. Fix such a $v$ and let $\chi=v(\psi,\phi)$. Then $\chi$ is minimal. Let now $s=v(c,d)$. Then $\chi_{|Y_0}=s$. Consider 
$$
Y_1=\big\{z\in X:z\in\Sigma^1, \chi(z)=s\big\}.
$$
By \cite[Lemma 3.1]{ST} $Y_1$ has the local maximum property. Take now a linear function $l:\R^2\to\R$ which defines the tangent plane to $v$ at the point $(c,d,s)\in\R^3$. It has the properties
\begin{itemize}
\item[] $l(c,d)=s=v(c,d)$
\item[] $l(x,y)< v(x,y)$ for $(x,y)\neq(c,d).$
\end{itemize}
Let $\eta(z)=l\big(\psi(z),\phi(z)\big)$, $z\in X$. Then $\eta$ is a minimal function as well and $\eta_{|Y_0}=s$. However, if $z\in Y_1\setminus Y_0$, 
$\big(\psi(z),\psi(z)\big)\neq(c,d)$ and then
$$
\eta(z)=l\big(\psi(z),\psi(z)\big)<v\big(\psi(z),\psi(z)\big)=\chi(z).
$$
So, $\eta$ is a plurisubharmonic function on $Y_1$, $\max\limits_{Y_1}\eta=s$ while
$$
\big\{z\in Y_1:\eta(z)=s\big\}=Y_0.
$$
Now Lemma \ref{lmp1} implies that $Y_0$ is a local maximum set.
\end{proof} 
\begin{proof}[Proof of Theorem \ref{lmp}] If $\phi$ is a minimal function then $\psi:=u+\phi$ is also minimal. Let $c,d$ be real constants. By Lemma \ref{lmp2} the following set
$\{\psi=c+d\}\cap\{\phi=d\}$ has the local maximum property if non-empty; since
\begin{eqnarray*}
&\{u=c\}\cap\{\phi=d\}\cap\Sigma^1=&\\
&\{\psi=c+d\}\cap\{\phi=d\}\cap\Sigma^1
\end{eqnarray*}
it follows that $\{u=c\}\cap\{\phi=d\}\cap\Sigma^1$ has the local maximum property if non-empty. 

But for fixed $c$ 
$$
\{u=c\}\cap\Sigma^1=\bigcup_{d\,\in\,\R}\{u=c\}\cap\{\phi=d\}\cap\Sigma^1
$$
and so it has the local maximum property as union of sets with the local maximum property (since it is closed).  
\end{proof}
In the sequel, the function $\alpha$ is assumed to satisfy the condition \eqref{Osric}.
 \subsection{Levi flat levels}
Let $W$ be a complex surface. In the sequel by {\em complex curve} of $W$ we mean a purely $1$-dimensional immersed complex space $C\hookrightarrow W$. We observe that a complex curve lying on a smooth hypersurface of $W$ must be regular (cfr. \cite[Proposition 4.2]{ST}). If $Y$ is the zero set of a smooth function $f:W\to\R$ or an analytic subset of $W$ we denote $Y_{\textrm {reg}}$ the regular part of $Y$, i.e.
$$Y_{\textrm{reg}}=\{z\in Y\ :\ \mathrm{d}f(z)\neq 0\}\;.$$

If $\Sigma^1\neq\varnothing$, according to \cite[Lemma 4.1]{ST}, some levels of the minimal function $\phi$ contain complex curves. We can expect that the same is true for the function $\alpha$. The results of this subsection go in this direction.

\begin{propos}\label{Jenny}
Let $X$ be a complex surface with a real analytic plurisubharmonic exhaustion function $\alpha:X\to\R$ and $Y$ a connected component of a level set $\{\alpha=c\}$ such that $Y\cap \Sigma^1\neq\varnothing$. Then, for every point $p\in Y_{\textrm{reg}}\cap \Sigma^1$ there exist an open neighbourhood $U\subset X$ and coordinates $z,w$ on $U$ such that $U\cong\D_z\times \D_w$ and
$$
U\cap Y\cap\Sigma^1\cong\bigcup_{t\in T}\{(z,f_t(z))\ \vert\ z\in\D_z\}
$$
where each $f_t:\D_z\to\D_w$ is a holomorphic function.
\end{propos}
\begin{proof}
The proof repeats verbatim from \cite[Lemma 4.1]{ST} once we observe that, by Theorem \ref{lmp}, the set $\Sigma^1\cap Y$ has the local maximum property.
\end{proof}

\begin{teorema}\label{LUCY}
Let $X$ be a complex surface with a real analytic plurisubharmonic exhaustion
function  $\alpha:X\to \R$ and $Y$ a connected component of a level set
$\{\alpha=c\}$ such that $Y\cap\Sigma^1\neq\varnothing$. Assume that $Y$ contains, at worst, isolated critical points of $\alpha$ and no compact complex curve. Then
\begin{itemize} 
\item[a)] $Y_{\textrm {reg}}$ is a real analytic Levi flat hypersurface. 
\end{itemize}
In general, if $Y$ does not contain local minimum points of $\alpha$ and $Y\subset\Sigma^1$ then
$$
\p\alpha\wedge{\bar\p}\alpha\wedge\p{\bar\p}\alpha=0,\>\p{\bar\p}\alpha\wedge\p{\bar\p}\alpha=0
$$
on $X$ and so 
\begin{itemize}
\item[b)] $\Sigma^1=X$,  all non-critical level sets of $\alpha$ are
compact Levi flat hypersurfaces and the regular parts of the critical level sets
have complex foliation.
\end{itemize}
\end{teorema}
\begin{proof}
Since, by Theorem \ref{lmp}, $Y\cap\Sigma^1\neq\varnothing$ has the local maximum property it cannot have isolated points hence $Y_{\textrm {reg}}\cap\Sigma^1\neq\varnothing$.  In view of Proposition \ref{Jenny}, every point $p$ of $Y_{\textrm {reg}}\cap\Sigma^1$ has a neighbourhood $U$ in which $Y\cap \Sigma^1$ can be expressed as a union of analytic discs:
$$
U\cap Y\cap\Sigma^1\cong\bigcup_{t\in T(p)}\{(z,f_t(z))\ \vert\ z\in\D_z\}
$$
where each $f_t:\D_z\to\D_w$ is a holomorphic function. Let us first suppose that there is a point $p$ such that $T(p)$ is infinite for a fundamental system of neighborhoods of $p$, take $U$ one of such neighborhoods and let 
$$
U\cap Y\cap \Sigma^1\supset\bigcup_{n\in\N}\{(z,f_n(z))\ \vert\ z\in\D_z\}
$$
with $f_n:\D_z\to\D_w$ a countable family of holomorphic functions with disjoint graphs. On every such graph we have that $\p{\bar\p}\alpha$ degenerates and, being those graphs integral curves for the kernel of such Levi form, we actually have that
$$\p{\bar\p}\alpha\wedge \p\alpha\wedge{\bar\p}\alpha=0$$
on each of those graphs. By analyticity, $\p{\bar \p}\alpha\wedge\p\alpha\wedge{\bar\p}\alpha=0$ on the whole of $Y_{\textrm{reg}}$, thus giving a foliation of it in complex curves, hence Levi flatness.

Suppose, at contrary, that for every point $p\in Y_{\textrm {reg}}\cap\Sigma^1$ there is only a finite number of those analytic disks composing $Y_{\textrm {reg}}\cap\Sigma^1$ which pass through $p$; then it is not hard to see that, starting from one of such disks, by analytic continuation, we get a complex curve $C$ embedded in $Y_{\textrm {reg}}$. The closure $\oli C$ of $C$ in $Y$ is a compact complex curve: this is clear if $\oli C\subset Y_{\textrm {reg}}$ otherwise we apply Remmert-Stein theorem (cfr. \cite[Chapter VII, Theorem 1]{N}) the critical points in $Y$ being isolated. Since by hypothesis $Y$ has no compact complex curve the only possible case is the previous one. This shows part a).

Assume now that $Y\subset\Sigma^1$ does not contain local minimum points of $\alpha$. Then, by Theorem \ref{lmp} $Y$ has the local maximum property and so \cite[Theorem 3.9]{ST} applies: there is an $s<c$ such that, if $K$ is the connected component of the set $\{s\le\alpha\le c\}$
containing $Y$, then its topological boundary ${\rm b}K$ is contained in $\{\alpha=s\}\cup
K$ and furthermore $\Int{K}$ is nonempty, the forms 
\begin{itemize}
\item $(\p\bar\p\alpha)\wedge
\p\alpha\wedge \bar\p\alpha$
\item $(\p\bar\p\alpha)\wedge \p\alpha$
\item $(\p\bar\p\alpha)\wedge \bar\p\alpha$
\item $\p{\bar\p}\alpha\wedge\p{\bar\p}\alpha$ 
\end{itemize}
vanish on $K\smallsetminus Y$.
Now $K\smallsetminus Y$ has nonempty interior in $X$, and so
$\p{\bar\p}\alpha\wedge\p\alpha\wedge{\bar\p}\alpha=0$ in $X$, by real analyticity. This shows that $\Sigma^1=X$.

Levi flatness of the regular parts of levels follows from Proposition \ref{Jenny}. 
\end{proof}
\begin{lemma}\label{OLD}
Let $Y$ be a (non necessarily compact) real analytic Levi flat hypersurface in a
complex surface $W$. Assume that there is a real analytic plurisubharmonic function
$\beta:V\to\R$ on an open neighbourhood $V$ of $Y$, such that $\{p\in
V:\beta(p)=0\}=Y$  and $\partial \beta(p)\not=0$ for all $p\in Y$. Then there is an open neighbourhood $U$ of $Y$, $U\subset V$, and a pluriharmonic function
$\chi:U\to\R$ such that $Y=\{p\in U:\chi(p)=0\}$ and
$\partial\chi(p)\neq 0$  for all $p\in Y$. 
\end{lemma}\begin{proof}
It is well known that a real analytic Levi flat hypersurface $M$ in a complex manifold admits a real analytic defining function which is pluriharmonic in a neighborhood of $M$ if and only its Levi foliation is defined by a nonvanishing real analytic closed 1-form (e.g cfr. \cite{barr}). 

In our situation since $Y$ is real analytic the foliation of $Y$ extends holomorphically on a neighbourhood of $Y$ (cfr. \cite{R}): there is an atlas for a neighbourhood of $Y$ in $X$, consisting of distinguished coordinate charts $(U_j,z_j,w_j)$, $p\mapsto
(z_j(p),w_j(p))$, 
$p\in U_j$, $(z_j(p),w_j(p))\in V_j\times B_j\subset\C^2$, $V_j$, $B_j$ connected, such that 
$$
Y\cap U_j=\big\{p\in U_j\!:{\rm Im}\,w_j(p)=0\big\}\;.
$$

Denote further $u_j={\rm Re}\,w_j$, $v_j={\rm Im}\,w_j$. Then $u_j, v_j\!:U_j\to\R$
are  pluriharmonic functions and $U_j\cap Y=\{v_j=0\}$. By this setup and the 
assumptions of the Lemma ($\partial\beta\not=0$ on $Y$) there are real 
analytic functions $\rho_j\!:U_j\to\R$ such that: $\beta=v_j\rho_j$ in $U_j$, 
${\rho_j}_{\vert U_j\cap Y}>0$.

We draw now some consequences for $\rho_j$ from the plurisubharmonicity of 
$\beta$.

By calculations on $\beta,\rho_j$ in the local coordinates 
$z_j,w_j=u_j+iv_j$, we obtain that at points $(z,u,0)\in U_j\cap Y$ the Levi form reduces to (we drop the index $j$ to simplify the notation)
\begin{eqnarray*}
\partial\bar\partial\beta(z,u,0)&=& -\frac{i}{2}\rho_z(z,u,0)dz\wedge d\bar 
w+\\
&+&\frac{i}{2}\rho_{\bar z}(z,u,0)dw\wedge d\bar z+\\
&+&\frac{i}{2}(\rho_w-\rho_{\bar w})(z,u,0)dw\wedge d\bar w.
\end{eqnarray*}
Since it is positive semidefinite in $U_j$ and the term $dz\wedge d\bar z$ is missing, we must have
$$
\rho_z(z,u,0)=0=\rho_{\bar z}(z,u,0),
$$
$(z,u,0)\in U_j\cap Y$ and since $V_j$, $j\in I$ are assumed connected 
$\rho_j(z,u,0)$ is constant in $z$. We denote it by 
$\rho^\ast_j(u)=\rho_j(z,u,0)$; $\rho_j\!:B_j\cap\R\to\R$ is a positive real analytic function. 

Since we have
$$
\partial\beta=v\rho_zdz+\left(-\frac{i}{2}\rho+v\rho_w\right)dw,
$$
we obtain eventually 
$$
\eta:=i\partial\beta_{|U_j\cap Y}=\frac{1}{2}\rho^\ast_j(u_j(p))du_j(p),
$$
for $p\in U_j\cap Y$.
$\eta$ is a nonvanishing real analytic closed 1-form which defines the Levi foliation of $Y$. This ends the proof of the lemma.
\end{proof}

We conclude this section by showing that, in absence of compact complex curves, the regular level sets of $\alpha$ intersecting $\Sigma^1$ are foliated by dense complex curves.

\begin{corol}\label{Lucy2}
Let $X$ be a complex surface with a real analytic plurisubharmonic exhaustion
function  $\alpha:X\to \R$ and $Y$ a connected component of the regular level set $\{\alpha=c\}$ such that $Y\cap\Sigma^1\neq\varnothing$. Assume that $Y$ does not contain compact complex curves. Then $Y$ is Levi flat, the leaves of the Levi foliation are dense in $Y$ and $Y\subseteqq\Sigma^1.$
\end{corol}
\begin{proof} $Y$ is Levi flat in view of Theorem \ref{LUCY} so, by Lemma \ref{OLD}, there are a neighbourhood $U$ of $Y$ and a pluriharmonic function $\chi:U\to\R$ such that $\{\chi=0\}=Y$ and $d\chi\neq 0$ on $Y$. The maximal complex subspace of $T_pY$, for $p\in Y$, is given by $\ker d\chi\cap \ker d^c\chi$; therefore, on the manifold $Y$, the foliation induced by Levi flatness is given by the $1$-form $d^c\chi$. Such a form is closed, because $dd^c\chi=0$ by pluriharmonicity, so every leaf of the foliation has trivial holonomy; hence, by a result of Sacksteder (cfr \cite[Notes to Chapter V, (2), p. 109]{CaNe}), the foliation has no exceptional minimal set then, as there are no compact (i.e. closed) leaves,  all leaves are dense in $Y$ (we refer to \cite{CaNe}, in particular to Chapter III, for the relevant definitions and results related to exceptional minimal sets of foliations).

This easily implies the last part of the statement: let $\psi$ be any plurisubharmonic function on $X$ and let $p_0\in Y$ be such that $\psi(p_0)=\max\limits_Y \psi$; then $\psi$ is constant on the leaf which passes through $p_0$, but then, by density, it is constant on $Y$. Therefore $Y\subseteqq\Sigma^1$.
\end{proof}

\section{Propagation of compact complex curves}\label{prph}\noindent
We observe the following fact. If $\chi : W \to\R$ is a nonconstant pluriharmonic function on a complex surface W such that the level set $S = \{p\in W :\chi(p) = 0\}$ is nonempty, compact and connected. Then $Y:=S\cap {\rm  Crt}(\chi)$ is a complex, compact analytic subset. Indeed, we know that $Y$ is a compact real analytic subset. Suppose that it is positive dimensional at some point $x\in Y$ and let $B$ an open ball centered at $x$ and $\tau : B \to \R$ a pluriharmonic conjugate of $\chi_{|B}$. Since $\chi_{|B\cap Y} = 0$ and the points of $Y$ are critical, the function $\tau$ is constant on each connected component $Y_0$ of $B\cap Y$ and we may assume $\tau_{|Y_0} = 0$. If $F : B \to\C$ denotes the holomorphic function $\chi + i\tau$ we then have $Y_0 = \{p \in B : \p F (p) = 0\}$.

\begin{lemma}\label{PL2bis}
Let $\chi\!:W\to \R$ be a pluriharmonic function on a complex surface $W$ with no critical points in $W$. Assume that $W$ contains a connected compact complex curve $C$. Then there exist a neighborhood $V$ of $C$ and proper holomorphic function $G:V\to\C$ such that $C=\{G=0\}$.
\end{lemma}
\begin{proof}
The restriction of $\chi$ to $C$ has to be a constant value, say $d$.

Choose a good covering $\left\{V_j\right\}_{j=1}^n$ of $C$ as in Lemma \ref{gocov}. Consider $\chi_{_{\vert V_j}}$, $j$ fixed. Assuming $r>0$ small enough
so that $V_j$ is contained in a topological ball in $W$, we conclude
that $\chi$ has a pluriharmonic conjugate in this ball, and so in
$V_j$, say $\tau'_j\!:V_j\to\R$. Since $\chi_{_{\vert C}}=c$, a
constant, we conclude that its harmonic conjugate on $C\cap V_j$ is locally constant. Since $C\cap V_j$
is connected ${\tau'_j}_{\vert C\cap V_j}$ is constant. Subtracting the latter constant from
$\tau'_j$ we obtain a function $\tau_j\!:V_j\to\R$ such
that
\vspace{3mm}
\begin{itemize}
\item[$\bullet$] $\tau_j$ is a pluriharmonic conjugate of $\chi_{\vert
V_j}$;\\
\item[$\bullet$] $\tau_{\vert C\cap V_j}\equiv
0.$
\end{itemize}
Consider now two intersecting neighborhoods  $V_j$, $V_k$ and define
$V_0\!:=V_j\cap V_k\neq\varnothing.$ Since $V_0$ is connected, $\tau_j-\tau_k=a$
constant in $V_0$ and so $\tau_j-\tau_k\equiv 0$ in $V_0$,
because ${\tau_j}_{\vert V_0\cap C}\equiv 0\equiv
{\tau_k}_{\vert V_0\cap C}$. (Note that 
$V_0\cap\,C\neq\varnothing$ if $V_0\neq\varnothing$.) Thus $\tau_j(p)=\tau_k(p)$, whenever $p\in V_j\cap V_k$. Consequently, the family $\{\tau_j\}_{j=1}^n$ defines a single-valued pluriharmonic function $\tau\!:\!V\to\!\R$, where $V\!\!=\bigcup\limits_{j=1}^nV_j$, such that $F(p)\!=\!\chi(p)+i\tau(p)$, $p\in V$, is holomorphic.

Therefore, there exists $F\in\mathcal{O}(V)$, a non constant holomorphic function, such that $F_{\vert _C}=d$. As $d$ is a regular value of $\chi$, $C$ is non singular, so is a connected component of $\{F=d\}$. It follows that there exists a neighborhood $V$ of $C$ such that $\{F=d\}\cap V=C$ and consequently we set $G=F-d$.\end{proof}

\begin{teorema}\label{BIR}Let $X$ be a weakly complete complex surface, $W\subset X$ a domain and $\chi:W\to \R$ be a pluriharmonic function. Suppose that a regular level of $\chi$ contains a compact complex  curve  Then $X$ is proper over a (possibly singular) complex curve.
\end{teorema}
\begin{proof}
Without loss of generality, suppose $C$ connected. In view of Lemma \ref{PL2bis} there is exist a neighborhood $V$ of $C$ and holomorphic function $G:V\to\C$ such that $C=\{G=0\}$. Then the family $\mathfrak F_0=\{G=\zeta, |\zeta|<\epsilon\}$, for some $\epsilon>0$, consists of compact complex curves so, by \cite[III.5.B]{Ni}, $\mathfrak F_0$ extends to a family $\mathfrak F$ globally defined on $X$ by a holomorphic map $\Phi:X\to R$, where $R$ is an open Riemann surface. In particular, $X$ admits a non constant holomorphic function. In view of \cite{Ohs}, $X$ is holomorphically convex hence proper over a Stein space (cfr. \cite{Car}).
\end{proof}

We have the following fundamental corollary
\begin{corol}\label{FUN}
Let $X$ be a complex surface, $\alpha$ a real analytic plurisubharmonic exhaustion function, $Y$ a compact connected component of the regular level set $\{\alpha=c\}$. Assume that $Y$ is Levi flat and contains a non closed leaf. Then there exist a pluriharmonic function $\chi:V\to\R$ on a connected neighbourhood of $Y$ and $\epsilon_0>0$ such that
\begin{itemize}
\item[{\rm a)}] $V\cap\{\alpha=c\}=Y,\>\> Y=\{\chi=0\}$
\item[{\rm b)}] the set                                    
$$                                                                                                                                                                                                                               
H=\left\{p\in V\!:0<\chi(p)<\epsilon_0\right\}\;,
$$
is relatively compact in $V$, does not contain a critical point of $\alpha$ or
$\chi$, and does not contain any compact complex curve;\\
\item[{\rm c)}] if $0<\epsilon<\epsilon_0$, the Levi flat hypersurface $\{\chi=\epsilon\}$ is foliated by dense complex leaves and $\alpha$ is constant on it;
\\
\item[{\rm d)}] $\p{\bar\p}\alpha\wedge\p{\bar\p}\alpha=0$ and
$\p\alpha\wedge{\bar\p}\alpha\wedge\p{\bar\p}\alpha=0$ on the whole of $X$;
\item[{\rm e)}] there is a real analytic function $\mu:H\to\R$ such that 
$$\p\bar\p\alpha=\mu\p\alpha\wedge\bar\p\alpha\quad\textrm{and}\quad d\mu\wedge d\alpha=0\;.$$
\end{itemize}
\end{corol}
\begin{proof}
a) It follows from Lemma \ref{OLD}.

b) By hypothesis, $\alpha$ has no singular point on $Y$, therefore there exists a neighborhood of it where no other singular points for $\alpha$ appear; on the other hand, the set ${\rm Crt(\chi)}$ of the critical point of $\chi$ has complex analytic structure so $\chi$ is constant on its connected component. This implies that also $\chi$ has no critical point on $Y$. Therefore, there exists $\epsilon_0$ such that the set $H$ does not contain critical points of either $\alpha$ or $\chi$.
 
Let us consider the open set 
 $$
 N=\{p\in V \ :\ -\epsilon_1<\chi(p)<\epsilon_0\}
 $$
 with $\epsilon_1$ so small that no critical point of $\alpha$ or $\chi$ is contained in $N$. 

 If $N\setminus Y$ contains a compact complex curve $C$, then 

 by Theorem \ref{BIR} $X$ is union of compact complex curves. This is absurd, as we supposed that $Y$ contained a non-closed complex leaf. Therefore, $H\subseteqq N\setminus Y$ does not contain any compact curve.
 
 c) Consider the Levi flat level $Y_\epsilon=\left\{\chi=\epsilon\right\}$, for $0<\epsilon<\epsilon_0$; by the previous point, $Y_\epsilon$ does not contain complex curves, so by Corollary \ref{Lucy2}, the leaves of the Levi foliation are dense in $Y_\epsilon$.
 Now, let $p^\ast\in Y_\epsilon$ be such that $\max\,\alpha_{\vert Y_\epsilon}=\alpha(p^\ast)$ and consider the leaf $F$ of the Levi foliation passing through $p^\ast$; the function $\alpha_{\vert F}$ is a plurisubharmonic function on $F$, which is a complex immersed curve, attaining maximum in an interior point, hence it is constant on $F$. But, $F$ being dense in $Y_\epsilon$, $\alpha$ has to be constant on $Y_\epsilon$.

d) Since $Y_\epsilon$ is a connected component of $\{\alpha=A_\epsilon\}$) and is Levi flat we have on $Y_\epsilon$
   \begin{equation}\label{Fortinbras}\begin{cases}
			\p\alpha\wedge{\bar\p}\alpha\wedge\p{\bar\p}\alpha=0&\\
   \p{\bar\p}\alpha\wedge\p{\bar\p}\alpha=0&
				\end{cases}\end{equation}
			
But since the connected Levi flat components of level set of $\alpha$ cover $N_0$, the equations \eqref{Fortinbras} hold on $N_0$, a set with non empty interior. In view of real analytity of the forms \eqref{Fortinbras} we conclude that the equations\eqref{Fortinbras} hold on the whole of $X$.

e) By part c), $\alpha$ is constant on every level set of $\chi$
in $H$, and so there is a real-valued function $\varphi:(0,\epsilon_0)\to\R$ such
that $\alpha{|H}=\varphi\circ\chi_0$. Since $\alpha$ and $\chi_0$ are both real
analytic and $\p\alpha$, $\p\chi_0$ do not vanish in $H$ by part b), it is clear that $\varphi$ is real analytic
and $\varphi'(t)\neq 0$ in $(0,\epsilon_0)$. Now, by direct computation
$$
\p\alpha\wedge{\bar\p}\alpha=(\varphi'\circ\chi_0)^2\p\chi_0\wedge{\bar\p}\chi_0,
$$
in $H$,
$$
\p{\bar\p}\alpha=(\varphi''\circ\chi_0)^2\p\chi_0\wedge{\bar\p}\chi_0,
$$
in $H$.

Thus we set $\mu={(\varphi/\varphi')}^2\circ\chi_0$, which is constant on the
level sets of $\chi$ in $H$ and therefore on those of $\alpha_{|H}$ as well.
Then $d\mu\wedge d\alpha=0$ in $H$. 
\end{proof}

We conclude this section with a description of the geometric structure of weakly complete surfaces.

\begin{teorema}\label{MAR1}
Let $X$ be a weakly complete complex surface and $\alpha\!:X\to\R$ a real analytic plurisubharmonic exhaustion function. Then three cases can occur:
\begin{enumerate}
\item[1)] $X$ is a modification of a Stein space;
\item[2)] $X$ is proper over a (possibly singular) complex curve;
\item[3)] the connected components of the regular levels of $\alpha$ are foliated with dense complex curves, i.e. $X$ is of Grauert type.
\end{enumerate}
\end{teorema}
\begin{proof}
Let us suppose that there exists a sequence of real numbers $\{c_n\}_{n\in\N}\subset\R$ tending to $+\infty$ such that
$$\{\alpha=c_n\}\cap\Sigma^1=\emptyset\;;$$
as $\Sigma^1$ is a closed set and the levels of $\alpha$ are compact, for every $n\in\N$ there exists $\delta_n>0$ such that for every $c\in(c_n-\delta_n,c_n+\delta_n)$ the intersection $\{\alpha=c\}\cap\Sigma^1$ is empty.

Let $\varphi$ be a minimal plurisubharmonic smooth exhaustion function for $X$; for every $n$ there is $\epsilon_n>0$ small enough such that
$$\{\alpha+\epsilon_n\varphi=c_n\}\subset\{c_n-\delta_n<\alpha<c_n+\delta_n\}$$
and then
$$\{\alpha+\epsilon_n\varphi=c_n\}\cap\Sigma^1=\emptyset\;.$$
The function $\alpha+\epsilon_n\phi$ is minimal as well, hence $X_n=\{\alpha+\epsilon_n\phi<c_n\}$ is a relatively compact strictly pseudoconvex domain (which we can suppose smoothly bounded up to some small perturbation of $c_n$), hence it is a modification of a Stein surface.
Moreover, $X_n$ is Runge in $X_{n+1}$, therefore $X$ itself is holomorphically convex and, possessing a plurisubharmonic function which is strictly plurisubharmonic at some point, it has to be a modification of a Stein space as well: this is the case 1).

If such a sequence does not exist, then there is $c_0\in\R$ such that for every $c>c_0$ the intersection $\{\alpha=c\}\cap\Sigma^1$ is not empty. 

Suppose that there exists $c_1>c_0$, regular value for $\alpha$, such that there is a connected componet $Y$ of $\{\alpha=c_1\}$ which does not contain compact complex curves and such that $Y\cap\Sigma^1\neq\emptyset$; we apply Theorem \ref{LUCY}, obtaining that $Y$ is Levi flat and by Corollary \ref{FUN}, part d), we get that 
$\p\alpha\wedge \bar{\p}\alpha\wedge\p\bar{\p}\alpha=0$
on the whole of $X$, hence every regular level of $\alpha$ is Levi flat and by Lemma \ref{OLD} every such level has a neighbourhood where it is given as the zero of a pluriharmonic function
Therefore, no regular level can contain a compact complex curve, otherwise, by Theorem \ref{BIR}, every level would, so all the regular levels are Levi flat and containing no complact complex curves, hence by Corollary \ref{FUN} part c), their connected components are foliated with dense complex leaves. This is case 3)

If every regular level  $\{\alpha=c\}$, for $c>c_0$,  contains a compact complex curve, then $X$ contains uncountably many compact complex curves; by \cite[Proposition 9 and 7]{Ni} there exist $V$ a neighbourhood of one of these curves and $f:V\to\C$ a holomorphic function which induces on $V$ a foliation in compact curves. Applying Theorem \ref{BIR} with $V=W$ and $\chi=\mathsf{Re}f$, after shrinking $V$, if needed, to avoid critical points for $f$, we conclude that $X$ is proper over a non-compact (possibly singular) complex curve, which is case 2).\end{proof}

It is easy to show that in case 1), $\Sigma^1$ is a (at most) countable union of compact complex curves (the exceptional divisor of the modification), whereas in case 2) and 3) it is obvious that $\Sigma^1=X$. We note that we can tell apart these two cases quite easily by looking at  global holomorphic functions: in case 3), $\mathcal{O}(X)=\C$, wherease in case 2) there always exist global non-constant holomorphic functions.

\begin{corol}\label{MAR2}
Let $X$ be a weakly complete $2$-dimensional normal complex space and $\alpha\!:X\to\R$ a real analytic plurisubharmonic exhaustion function. Then three cases can occur:
\begin{enumerate}
\item[1)] $X$ is a modification of a Stein space;
\item[2)] $X$ is proper over a (possibly singular) complex curve $C$;
\item[3)] $X$ is of Grauert type.
\end{enumerate}
\end{corol}
\begin{proof}
Desingularize $X$ (cfr. \cite{Hir}) and apply Theorem \ref{MAR1} to the desingularization. 
\end{proof}

\section{Existence of proper pluriharmonic functions - I}\label{EPPAF}
\noindent Trough Sections 5, 6 we assume that $X$ is a Grauert type surface, i.e. $X$ satisfies conditions of case 3) of Theorem \ref {MAR1}.

We want to prove the following
\begin{teorema}\label{SCOTTISH}
Let $X$ be a Grauert type surface with a real analytic plurisubharmonic exhaustion function $\alpha$. In particular $\Sigma^1=X$ and
$$
\p\alpha\wedge\bar\p\alpha\wedge\p\bar\p\alpha=\p\bar\p\alpha\wedge\p\bar\p\alpha=0.
$$ 
If $\rdim {\rm Crt}(\alpha)\le 2$, let $Z$ be the absolute minimum set of $\alpha$, then 
\begin{itemize}
\item[1)] $Z$ is the union of finitely many complex curves;
\item[2)] there exists an increasing convex function $\lambda$ such that $\chi=\lambda\circ(\alpha_{\vert X\setminus Z})$ is pluriharmonic and proper.

\end{itemize}
\end{teorema}
From the existence of $\chi$ immediately follows that
\begin{corol}\label{SCOTTISH1}
The function $\alpha$ does not have local minimum points other then absolute minimum ones. The level sets of $\alpha$, except absolute minimum ones, have pure dimension {\rm 3}.
\end{corol}
In the remainder of this section we prove some auxiliary lemmas.

\begin{lemma}\label{LAURA2}. Let $W$ be a complex surface, and $\beta:W\to\R$ a real analytic
function, such that
$$
\p\beta\wedge{\bar\p}\beta\wedge\p{\bar\p}\beta=0,\>\>
\p{\bar\p}\beta\wedge\p{\bar\p}\beta=0
$$
in $W$. Then there are a real analytic function $\mu:W\smallsetminus {\rm Crt}(\beta)\to\R$, such that
$$
\p{\bar\p}\beta=\mu\,\p\beta\wedge{\bar\p}\beta
$$
in $W\smallsetminus{\rm Crt}(\beta)$.
 \end{lemma}
\begin{proof}
Observe that we do not assume that $\beta$ is plurisubharmonic. The real-analiticity property of $\mu$ is obviuos, once we establish the elementary fact that for every $p$ such that $\p\beta(p)\neq 0$ there is
$\mu(p)\in\R$ satisfying
$$
\p\bar\p\beta(p)=\mu(p)(\p\beta\wedge\bar\p\beta)(p).
$$
We sketch the details. Point $p$ is fixed. If $\tilde{z},\tilde{w}$ are any complex local coordinates near $p$, the complex Hessian of $\beta$ can be diagonalized by a linear change of coordinates $\tilde{z},\tilde{w}$ into some $z,w$, so that
$$
\p\bar\p\beta(p)=a(dz\wedge d\bar z)(p)+c(dw\wedge d\bar w)(p),
$$
with $a,c$ are real. We have 
$$
0=(\p\bar\p\beta\wedge \p\bar\p\beta)(p)=ac(dz\wedge d\bar z\wedge dw\wedge d\bar w)(p),
$$
$$
\p\beta(p)=(rdz+sdw)(p), r,s\in\C, (r,s)\neq (0,0),
$$
$$
0=(\p\bar\p\beta\wedge \p\bar\p\beta)(p)=(r\bar rc+s\bar sa)(dz\wedge d\bar z\wedge dw\wedge d\bar w)(p),
$$ 
so $ac=0$ and $(r\bar rc+s\bar sa)=0.$

If $a=0,b\neq 0$ then $r\bar rc=0$, $r=0$. Thus
$$
d\beta\wedge d\bar \beta)(p)=s^2(dw\wedge d\bar w)(p),\, s\neq 0,
$$
and so:
$$
(\p\bar\p\beta)(p)=c(dw\wedge d\bar w)(p)=\mu(d\beta\wedge d\bar \beta)(p)
$$
with $\mu(p)=c/s^2.$

The case $a\neq 0$, $c=0$ is completely analogous, with roles of $z$, $w$ being interchanged.

If $a=c=0$, take $\mu(p)=0.$
 \end{proof}
\begin{lemma}\label{LIS}
Let $W$ be a connected complex surface and $\beta:W\to\R$,
$\mu:W\smallsetminus {\rm Crt}(\beta)\to\R$  two real analytic functions such that 
\begin{enumerate}
\item[1)] ${\rm dim}_\R\,{\rm Crt}(\beta)\le 2$;
\item[2)] $\p{\bar\p}\beta=\mu \p\beta\wedge{\bar\p}\beta$, on
$W\!\smallsetminus\!{\rm Crt}(\beta)$; 
\item[3)] $d\mu\wedge d\beta=0$ on $W\smallsetminus {\rm Crt}(\beta)$;
\item[4)] $\beta(W)=\beta(W\smallsetminus {\rm Crt}(\beta))$.
\end{enumerate}
Then there is a nonconstant pluriharmonic function $\chi:W\to\R$ and a real
analytic function $\theta:\beta(W)\to\R$ such that
\begin{enumerate}
\item[i)] $\theta'(t)>0$ for $t\in\beta(W)$;
\item[ii)] $\chi=\theta\circ\beta$ (and so $d\chi\wedge d\beta=0$ in $W$);
\item[iii)] any pluriharmonic function $\chi^\ast:W\to\R$ such that ${\rm d}\chi\ast\wedge {\rm d}\beta=0$ on $W$, must be of the form $\chi^\ast=c\chi+c_1$, where $c, c_1$ are real constants.
\end{enumerate}
\end{lemma}
\begin{proof}
The following is well known. 

{\sf Assertion 1}. If $d\mu\wedge d\beta=0$ in $W_0=W\smallsetminus {\rm Crt}(\beta)$,
then $\mu$ is constant on each connected component of every level set of
$\beta$.

{\sf Assertion 2}. If ${\rm Crt}(\beta)$ has topological dimension $\le 2$, then
there is a real analytic function $m:\beta(W_0)\to\R$ such that
$\mu(p)=m(\beta(p))$, $p\in W_0$. If, in addition, $\beta$ does not have local
minimum or local maximum points in $W$, then $\beta(W)=\beta(W_0)$ and 
$\p{\bar\p}\beta=(m\circ\beta)\p\beta\wedge{\bar\p}\beta$ in $W$.

Observe first, that since $W_0$ does not contain critical points, and since it
is connected due to the fact that ${\rm dim}_\R\left({\rm Crt}(\beta)\right)\le 2$,
$\beta(W_0)$ is an open (perharps unbounded) interval, say
$\beta(W_0)=(a,b)\subseteq (-\infty,+\infty)$. Fix a point $p_0\in W_0$. Consider the
family of open intervals $\mathcal I\subset (a,b)$ such that there is a real
analytic function $m_{\mathcal I}:{\mathcal I}\to\R$ satisfying 
$m_{\mathcal I}\circ\beta_{|W_{{\mathcal I}}}=\mu_{|W_{\mathcal I}}$, where 
$W_{{\mathcal I}}$ is the connected open component of the open set 
$\{p\in W_0:\beta(p)\in {\mathcal I}\}$ that contains $p_0$. It is evident
that for fixed ${\mathcal I}$ function $m_{\mathcal I}$ is unique, and so if 
${\mathcal I}_1\subset {\mathcal I}_2\subset (a,b)$ then 
${m_{\mathcal I_2}}_{|{\mathcal I_1}}=m_{{\mathcal I_1}}$. Hence, by Zorn
Lemma, there is a maximal interval with this property, assuming that there is any
nonempty interval. Denote it by ${\mathcal I}^\ast(a^\ast,b^\ast$ and ${\mathcal I}^\ast(b^\ast,b^\ast)=\varnothing$ where $b^\ast=\beta(p_0)$ if no interval $\mathcal I$ exists. 

We claim that $(a^\ast,b^\ast)=(a,b)$. Suppose not and assume, without loss of generality, that $b^\ast <b$. (The case $a^\ast>a$ is analogous). (This also
covers the case when
$a^\ast=b^\ast$, ${\mathcal I}^\ast=\varnothing$.) Choose a point 
$p^\ast\in\{\beta=b^\ast\}\cap\overline{W_{{\mathcal I}}}\cap W_0$, in
particular $p^\ast\in b_{W_0}W_{\mathcal I}$ (the boundary of 
$W_{\mathcal I}$ in $W_0$). In case ${\mathcal I}^\ast=\varnothing$, we choose
$p^\ast=p_0$.

Since $\p\beta(p^\ast)\neq 0$, we can select a local (real)
coordinate system at $p^\ast$, say $(X_1(p),X_2(p),X_3(p),X_4(p))$ on a
neighbourhood $N_0$ such that $X_1(p)=\beta(p)$ and 
$X_1(p^\ast)=X_2(p^\ast)=X_3(p^\ast)=X_4(p^\ast)=0$. We can select an
$\varepsilon>0$ and a smaller neighbourhood $N$, $p^\ast\in N\subset N_0$, such
that
$$
\{(X_1(p),X_2(p),X_3(p),X_4(p)):p\in
N\}=(b^\ast-\epsilon,b^\ast+\epsilon)\times (-\epsilon,\epsilon)^3.
$$
Since $\{\beta=t\}\cap N=\{t\}\times (-\epsilon,\epsilon)^3$ is connected, and $d\mu\wedge
d\beta=0$, we obtain by Assertion 1 that $\mu_{|\{\beta=t\}\cap N}$ is constant
and so $\mu_{|\{\beta=t\}\cap N}=m_\epsilon(t)$. Since, in these
coordinates $m_\epsilon(t)=\mu(t,0,0,0)$, we obtain $m_\epsilon$ is analytic on 
$(b^\ast-\varepsilon,b^\ast+\varepsilon)$ and $m_\epsilon\circ\beta_{|N}=\mu_{|N}$.

In case ${\mathcal I}^\ast=\varnothing$ and $a^\ast=b^\ast$, $p^\ast=p_0$ thus yelds a
nonempty interval ${\mathcal I}^\ast=(b^\ast-\epsilon,b^\ast+\epsilon)$. Since
$m_\epsilon\circ\beta=\mu$ on $N$, and $\mu_\epsilon\circ\beta$, $\mu$ are analytic
functions, the identity must hold on $W_{(b^\ast-\epsilon,b^\ast+\epsilon)}$ as well. In
case $a^\ast<b^\ast$ we obtain 
$m_{\mathcal I}\cap N=(b^\ast-\epsilon,b^\ast+\epsilon)\times (-\epsilon,\epsilon)^3$ (using the local
coordinate system), and so $m_{\mathcal I}(t)=\mu(t,0,0,0)=m_\epsilon(t)$ for 
$t\in{\mathcal I}^\ast\cap (b^\ast-\epsilon,b^\ast+\epsilon)=(b^\ast-\epsilon,b^\ast)$. Thus 
$m_{{\mathcal I}^\ast}$ and $m_\epsilon$ define consistently a real analytic
function, call it $m_{{\mathcal I}_1}:{\mathcal I}_1\to\R$, where 
${\mathcal I}_1=(a^\ast,b^\ast+\epsilon)$. 
Since $m{{\mathcal I}_1}\circ\beta_{|W_{\mathcal I}}=\mu_{|W_{\mathcal I}}$,
and both $m{{\mathcal I}_1}$ and $\mu$ are defined and real analytic in 
$W_{{\mathcal I}_1}$, that contains $W_{\mathcal I}$, we obtain 
$m{{\mathcal I}_1}\circ\beta=\mu$ in $W_{{\mathcal I}_1}$.

We conclude that ${\mathcal I}^\ast=(a,b)$. Since $W_{(a,b)}=W_0$, we have
$\mu=m\circ\beta$ in $W_0$, the first of the Assertion 2.

Furthermore, if $\beta(W)=\beta(W_0)$, $m\circ\beta$ is defined and real
analytic in $W$, and since 
$\p{\bar\p}\beta=(m\circ\beta)\p\beta\wedge{\bar\p}\beta$ in $W_0$ as already
shown, the same identity must hold in $W$ (just by continuity).

This completes the proof of Assertion 2.

By Assertion 2, any (pluriharmonic) function $\chi^\ast:W\to\R$ such that 
${\rm d}\chi^\ast\wedge d\beta=0$ in $W$ must be of the form 
$\chi^\ast=\theta_1\circ\beta$, $\theta_1$ real analytic on $\beta(W)$. We
look now for a condition for $\theta_1$ so that $\theta_1\circ\beta$ be
pluriharmonic.
\begin{eqnarray*}
0=\p{\bar\p}(\theta_1\circ\beta)\!\!&=&\!\!(\theta_1''\circ\beta)
\p\beta\wedge{\bar\p}\beta+(\theta_1'\circ\beta)\p\wedge{\bar\p}\beta=\\ 
&=&\!\!(\theta_1''\circ\beta)\p\beta\wedge{\bar\p}\beta+
(\theta_1'\circ\beta)\p\wedge{\bar\p}\beta
\end{eqnarray*}
i.e.
$\left[(\theta_1''\circ\beta)(p)+(m\circ\beta)(p)\theta_1'(\beta(p))\right]
(\p\beta\wedge{\bar\p}\beta)(p)=0$.

Since $(\p\beta)(p)\neq 0$ for $p\in W\smallsetminus {\rm Crt}(\beta)=W_0$, we obtain the
condition for $t=\beta(p)$,
\begin{equation}\label{Claudius}\theta_1''(t)+m(t)\theta_1'(t)=0\;.\end{equation}
Since 
$$
\{t=\beta(p),\, p\in W\smallsetminus {\rm Crt}(\beta)\}=\beta(W)=(a,b)
$$
we have to solve \eqref{Claudius} on $(a,b)$. Applying standard techniques we
obtain, fixing a point $t_0\in (a,b)$
$$
\theta_1(t)=c\int_{t_0}^t\exp\left(-\int_{t_0}^\tau m(\sigma)\,d\sigma\right)+
c_1.
$$
Choosing as $\theta$ solution with $c=1$ and $c_1=0$ we have evidently
$\theta'(t)>0$ on $(a,b)$, $\chi=\theta\circ\beta$ satisfies ${\rm d}\chi\wedge
{\rm d}\beta=0$ and is pluriharmonic, and any other pluriharmonic solution
$\chi^\ast$ satisfying $d\chi^\ast\wedge
d\beta=0$, is equal to 
$$
\chi^\ast=\theta_1\circ\beta=(c\theta+c_1)\circ\beta=c\chi+c_1,
$$
as required.
\end{proof}
\begin{lemma}\label{MAR}
Let $X$ be a complex surface, $\alpha$ a real analytic plurisubharmonic exhaustion function. Let $U\subset X$ be a domain such that ${\rm dim}_\R\,{\rm Crt}(\alpha)\cap U\le 2$. If $\alpha_{\vert U}$ has an absolute minimum value $A_0$, let $Z_0=\{p\in U:\alpha(p)=A_0\}$ (otherwise $Z_0=\varnothing$), then
\begin{itemize}
\item[{\rm i)}] $\alpha$ does not have local minimum points on $U\setminus Z_0$; 
\item[{\rm ii)}] there is a proper pluriharmonic function $\chi:U\smallsetminus Z_0\to\R$ such that $\alpha=\lambda\circ\chi$ for some increasing convex function $\lambda$;
\item[iii)] if $Z_0\neq \varnothing$ 
$$
\lim\limits_{U\ni p\to Z_0}\chi(p)=-\infty.
$$
\end{itemize}
\end{lemma}
\begin{proof}
Let $W=U\setminus Z_0$ and $W_0=U\setminus {\rm Crt}(\alpha)$. Of course ${\rm dim}_\R\,Z_0\le 2$  and $Z_0$is real analytic. Thus, $U$ being connected and 
${\rm dim}_\R\,U=4$, $Z_0$ cannot separate $U$ and ${\rm Crt}(\alpha)$ cannot separate $W$,
so $W$ and and $W_0$ are connected; it follows that
\begin{equation}\label{Yorick}\alpha(W_0)=\alpha(W)=(A_0,A^\ast)\qquad\textrm{where }A^\ast\le +\infty\;.\end{equation}
(Indeed, the only possible values in $\alpha(W)\smallsetminus\alpha(W_0)$ could be one of the
local minimum values but since $\alpha(W_0)$ is a connected interval and 
$\inf\,\alpha(W_0)=\inf\,\alpha(W)=A_0$, we have equality.)

Let $Y$ be a connected component of $\{\alpha=c\}$ intersecting $\Sigma^1$. In view of Theorem \ref{LUCY} $Y$ is Levi flat so, by Lemma \ref{OLD}, it is the zero set of a pluriharmonic function near $Y$ so the fundamental Corollary \ref{FUN} applies. Then, saving the same notations $V\supset Y$, $\chi:V\to\R$, $\overline{N}\subset V$,  by part d)
$$
\p{\bar\p}\alpha\wedge \p{\bar\p}\alpha=0,\,\p\alpha\wedge{\bar\p}\alpha\wedge \p{\bar\p}\alpha=0
$$
and hence, by Lemma \ref{LAURA2} there is a real analytic function $\mu:W_0\to\R$ such
that
\begin{equation}\label{Polonius}\p{\bar\p}\alpha=\mu\p\alpha\wedge{\bar\p}\alpha\;.\end{equation}
By part e) of Corollary \ref{FUN}, $d\mu\wedge d\alpha=0$ in $N$. But then by real analyticity, 

\begin{equation}\label{Laertes}d\mu\wedge d\alpha=0\qquad \textrm{in }X\setminus{\rm Crt}(\alpha)\;.\end{equation}

In view of \eqref{Yorick}, \eqref{Polonius}, \eqref{Laertes}, function $\alpha$ and sets $W$, $W_0$ satsfy all the conditions of Lemma
\ref{LIS} with $\beta=\alpha$ and so there exist real analytic functions
$\theta:(A_0,A^\ast)\to\R$, such that $\theta'(t)>0$, for all $t$, $\chi=\theta\circ\alpha:W\to\R$
is pluriharmonic, i.e. $\chi=\theta\circ\alpha:W\to\R$.

Let now $\lambda=\theta^{-1}:\chi(W)\to\R$. Of course, $\lambda$ is real analytic, $\lambda'(t)>0$
and $\alpha=\lambda\circ\chi$ in $W$. Since
$\p{\bar\p}\alpha=(\lambda''\circ\chi_0)\p\chi\wedge{\bar\p}\chi$, and $\alpha$ is
plurisubharmonic, $\lambda''\ge 0$, i.e. $\lambda$ is convex.This shows ii).

Observe now that, since $\chi$ is pluriharmonic on $W$ (by definition), $\chi$ does not
have any local minimum points in $W$, and since $\alpha=\lambda\circ\chi$, $\lambda$ strictly
increasing, also $\alpha$ cannot have any local minimum points in $W$.

The set $\chi(W)$ is an open interval, say $\chi(W)=(c_0,c^\ast)$, with $c_0\ge
-\infty$, $c^\ast\le+\infty$. Since $\alpha(p)\to 0$ as $p\to Z_0$, and $\chi=\theta\circ\alpha$,
we obtain
\begin{equation}\label{Horatio}\>\lim_{p\to Z_0}\chi(p)=c_0\;.\end{equation}
\noindent We prove now that $c_0=-\infty$. Suppose , to the contrary, that $c_0>-\infty$; we
will show that $Z_0$ is empty. Consider any regular point $p\in Z_0$ and its
neighbourhood $V$ with complex coordinates $z,w$. Without loss of generality, $V\subset \C^2$. Then $p$ is either an isolated point, or $T_p(Z_0)$, the tangent space, is a line
or a plane in $\C^2$. Whatever of these three cases may be, there is a complex
line $L$ through $p$ transversal to $T_p(Z_0)$, hence $p$ is an isolated point
of $L\cap Z_0\cap V$. The pluriharmonic function $\chi$ restricted to 
$L\cap Z_0\cap V$ is a bounded harmonic function with as its isolated singular
point and so extends to a harmonic function 
$\tilde\chi:L\cap\{\vert (z,w)-p\vert^2<\delta\}\to\R$,which must have value $c_0$
at $p$, by \eqref{Horatio}, and so a strict minimum there (as $p$ is an isolated point of
$Z_0=\{\alpha=\min\alpha\}$). This is a contradiction. Thus $c_0=-\infty$, which proves
iii).
\end{proof}

\begin{proof}[Proof of Theorem \ref{SCOTTISH}]
The existence of $\chi$ follows from Lemma \ref{MAR} applied to $U=X$. The absolute minimum set $Z$ of $\alpha$ is a compact complex pluripolar set in $X$ defined by the plurisubharmonic function
$$
\psi(p)=\begin{cases}
\chi(p)\>\>&{\rm if} \>\>x\in X\smallsetminus Z\\
-\infty\>\>&{\rm if} \>\>p\in Z.
\end{cases} 
$$
(i.e. $Z=\{\psi=-\infty\}$). Since, in addition, $Z$ is compact and real analytic from Lemma \ref{ATT} (see Appendix), it follows that $Z$ is the union of finitely many compact complex curves. Theorem \ref{SCOTTISH} is completely proved.
\end{proof}

\section{Existence of proper pluriharmonic functions - II}\label{mainp}\noindent
In this section, we prove the last part of the Structure Theorem, namely the following

\begin{teorema}\label{maint}
Let $X$ be a Grauert type surface and $\alpha: X\to\R$ a real analytic plurisubharmonic exhaustion. Assume that $\rdim {\rm Crt}(\alpha)=3$.Then either there is a proper pluriharmonic function $\chi:X\to\R$, or there is a double holomorphic covering map $\pi:X^\ast\to X$ and a proper pluriharmonic function $\chi^\ast:X^\ast\to\R$.
\end{teorema}

In order to present the proof of this result, we need a careful analysis of the critical level sets of $\alpha$, to which is devoted the reminder of this section. All the results before the proof of Theorem \ref{maint} assume implicitly its hypotheses.

\subsection{Geometrical structure of the set of critical points of $\alpha$}\label{gestr}

The critical set $\Crt(\alpha)$ is a real-analytic subset of $X$, and so, its regular part is dense in it (\cite{N}). The regular part is the union of (at most) countably many pairwise disjoint locally closed real-analytic submanifols of $X$.

Denote by $\{M^i\}_{i\in I}$ the collection of $3$-dimensional components of the regular part of $\Crt(\alpha)$ and, for each $i\in I$, let $A^i\in\R$ be the value of $\alpha$ on $M^i$; we set 
$$\widetilde{M}=\bigcup_{i\in I} \overline{M}^i\;.$$ 

Since we assume that $\rdim {\rm Crt}(\alpha)=3$, $\widetilde{M}\neq\emptyset$.

\begin{propos}\label{Desdemona}Let $U$ be a connected component of $X\setminus\widetilde{M}$, then
\begin{enumerate}
\item[1)] $U$ is open
\item[2)] if $M^i\cap {\rm b}U\neq\emptyset$, then $M^i\subseteqq {\rm b}U$
\item[3)] $\alpha\vert_{{\rm b}U}$ is constant and equal to $\min_{\overline{U}}\alpha$.
\end{enumerate}
\end{propos}
\begin{proof} We note that, $\alpha$ being proper, $\{\overline{M}^i\}_{i\in I}$ is a locally finite family of compact sets, hence their union $\widetilde{M}$ is a closed subset of $X$. Thus 1) follows.

To prove 2), we remark that each $M^i$ is relatively open in $\Crt(\alpha)$, so $M^i\cap \overline{M}^j=\emptyset$ whenever $i\neq j$; moreover, for each $i\in I$, there exists an open set $V^i\subseteqq X$ such that $\widetilde{M}\cap V^i=M^i$ and $V^i\setminus M^i$ has \emph{at most} $2$ connected components. 
Up to shrinking the open sets $V^i$, we can assume that $V^i\cap V^j=\emptyset$ if $i\neq j$; therefore
$$V^i\cap\widetilde{M}=M^i\quad\textrm{and}\quad V^i\setminus M^i\subseteqq X\setminus \widetilde{M}\;.$$
If, for some $i$, $M^i\cap {\rm b}U\neq \emptyset$, then $V^i\cap U\neq\emptyset$; therefore there exists a connected component $V^i_+$ of $V^i\setminus M^i$ which is contained, by connectedness, in $U$. Since the boundary of $V^i_+$ contains $M^i$, $M^i\subseteqq \overline{U}$.

We now show 3). Let
$$A=\min_{\overline{U}}\alpha\;.$$
Now, suppose that there exists a face $M^i$ of $U$ such that $A^i>A$. By Theorem \ref{MAR}, there exsits a real-analytic concave function $\theta:(A, A^*)\to\R$, with $A^*=\sup_U\alpha\in (A, +\infty]$, such that $\chi:=\theta\circ(\alpha\vert_{U\setminus\Min(\alpha)})$ is a pluriharmonic function in $U\setminus\Min(\alpha)$.

By the previous part of this proof, at least one component  $V^i_+$ of $V^i\setminus M^i$ intersects $U$; let $H$ be a connected component of $V^i_+\setminus\{\alpha=A^i\}$. If $\alpha\vert_H<A^i$, by Hopf Lemma \cite[Lemma 3.8]{ST}, $bH\cap \widetilde{M}=\emptyset$, but this cannot happen for all the connected components of $V^i_+\setminus\{\alpha=A^i\}$, as $bV^i_+\cap M^i\neq\emptyset$. Therefore we have that $\alpha\vert_H> A^i$ for some $H$, hence $A^*>A^i$.

Let $c=\theta(A^i)$, then $\chi\vert_{V^i_+}>c$ and $\chi\vert_{M^i}=c$. By Lemma \ref{Mario}, there exist a domain $\widetilde{V}\supseteq V^i_+\cup M^i$ and a pluriharmonic function $\widetilde{\chi}:\widetilde{V}\to\R$ extending $\chi$. We can assume, without loss of generality, that $\widetilde{V}$ is a connected open subset of $\{x\in  X \: \alpha(x)>A\}$ 

Consider now the function $\widetilde{\alpha}=\theta^{-1}\circ \widetilde{\chi}:\widetilde{V}\to\R$; such a function is real-analytic and coincides with $\alpha$ on the non-empty open set $V^i_+$, therefore they are equal on $\widetilde{V}$, by connectedness.

Since $\chi\vert_{V^i_+}>c$, $d\widetilde{\chi}(x)\neq 0$ for $x\in M^i$, but then 
$$d\alpha(x)=(\theta^{-1})'(x)d\widetilde{\chi}(x)\neq 0\;,$$ 
as $\theta^{-1}$ is increasing. 

This contraddicts the fact that $M^i\subseteqq \Crt(\alpha)$. Therefore, $A^i=A$.

\end{proof}

We call any connected component $U$ of $X\setminus\widetilde{M}$ a {\it cell}, and any $M^i$ a {\it face}. If $M^i\subset \oli U$, we call $M^i$ a {\it face} of $U$. To every face $M^i$ we associate an open neighbourhood $V^i$, as described in the proof of Proposition \ref{Desdemona}.

We note that, for every cell $U$, $\bar{U}\cap \widetilde{M}\neq\emptyset$, otherwise $U=\overline{U}$, i.e. $U=X$ by connectedness.

If $M^i$ is a face of $U$ and $V^i\setminus M^i\subseteqq U$, we call $M^i$ an \emph{internal face} of $U$ and set 
$$W^i=U\cup M^i\;.$$ 
Such $W^i$ is again an open set, as $W^i=U\cup V^i$, and it is uniquely determined by $M^i$, as $V^i\setminus M^i$ can be contained in a unique cell.

On the other hand, if $V^i\setminus M^i= V^i_+\cup V^i_-$ and $V^i_+\subseteqq U$, $V^i_-\subseteqq U'$ with $U\neq U'$, we say that $M^i$ is a \emph{connecting face} between $U$ and $U'$ and we set
$$W^i=U\cup M^i\cup U'\;.$$ 
Again, $W^i=U\cup V^i\cup U'$, so it is an open set and it is uniquely determined by $M^i$.

\begin{teorema}\label{Emilia}$\widetilde{M}=\{x\in X\ :\ \alpha(x)=\min_{X}\alpha\}$
\end{teorema}
\begin{proof}  
By part 3) of Proposition \ref{Desdemona}, $\alpha$ is constant on $bU$ for every cell $U$.

Let $\bar{A}$ be the smallest element of the set $\{A^i\}_{i\in I}$. By properness, there is a finite number of faces contained in $\{\alpha=\bar{A}\}$; let  us change the indexing so that these faces are $M^1,\ldots, M^N$.

We set 
$$W=\bigcup_{i\in I} W^i=X\setminus\bigcup_{i\in I}\left(\overline{M}^i\setminus M^i\right)$$
and we remark that, being the family of compact sets $\{\overline{M}^i\setminus M^i\}_{i\in I}$ locally finite, $W$ is open in $X$ and, as $\dim (\overline{M}^i\setminus M^i)\leq 2$ (as a semi-analytic set), $W$ is also connected.

We can also write
$$W=\bigcup_{i\in I} M^i\cup (X\setminus \widetilde{M})\;.$$
Now, define $H\subseteqq W$ as
$$H:=W^1\cup\ldots\cup W^N\;.$$
It is clear that $H$ is open in $W$; moreover,
 $$H=(M^1\cup\ldots\cup M^N)\cup H_0$$
 where $H_0$ is a union of finitely many cells $U$, all such that $\alpha\vert_{bU}= \bar{A}$.

Obviously, the relative boundary of such a cell $U$ in $W$ is contained in $M^1\cup\ldots \cup M^N$ and the closure of $H_0$ in $W$ is $H$. Therefore, $H$ is both closed and open in $W$ and, since $W$ is connected, $H=W$. Therefore, $\alpha\vert_{M^i}=\bar{A}$ for every $i\in I$; by continuity we have that $\alpha$ is constant on $\widetilde{M}$.

Now, let $U$ be a cell. Suppose $\{\alpha=\overline{A}\}$ intersects $U$, and let $Y$ be a connected component of $\{\alpha=\overline{A}\}$ such that $Y\cap U\neq\varnothing$. Since $Y\subset \Crt(\alpha)$ and $\dim_\R(Y\setminus\bigcup_iM^i)\leq 2$, we have $\dim_\R Y\cap U\leq 2$. Applying Lemma \ref{MAR} to the domain $U$, we obtain the existence of a pluriharmonic function $\chi:U\setminus Y\to \R$, of the form $\chi=\lambda\circ(\alpha\vert_{U\setminus Y})$, with $\lambda$ an increasing convex function, such that
$$\lim_{x\to U\cap Y}\chi(x)=-\infty\;.$$
In particular
\begin{equation}\label{Schiller}
\lim_{t\to\overline{A}^+}\lambda(t)=-\infty\;.\end{equation}
Let $M$ be a face adiacent to $U$ and $q\in M$ a point. Choose a connected open neighborhood $H$ of $q$ such that $H\setminus M$ has two connected components, at least one of them, denoted by $H^+$, contained in $U$. Define now function $\chi_1:H\to[-\infty,+\infty)$ by
$$
\chi_1(x)=\left\{\begin{array}{ll}\chi(x)&x\in H^+\\-\infty&x\in H\setminus H^+\;.\end{array}\right.$$
Then
$$\lim_{H\ni x\to M\cap H}\chi_1(x)=\lim_{H\ni x\to M\cap H}\lambda\circ\alpha(x)=\lim_{t\to\overline{A}^+}\lambda(t)=-\infty\;,$$
by \eqref{Schiller}. Thus $\chi_1$ is plurisubharmonic in $H$ and equal to $-\infty$ in $H\setminus H^+$, a set with non-empty interior, which is a contradiction. 

Therefore $\{\alpha=\overline{A}\}\cap U=\varnothing$. \end{proof}

We immediately obtain the following

\begin{corol} \label{Bianca}There are only finitely many faces and finitely many cells.\end{corol}
\begin{proof} The faces are the $3$-dimensional components of the regular part of $\alpha^{-1}(\bar{A})$, which are finitely many, as $\alpha$ is proper. Moreover, every face may belong to at most two cells, so also the cells are finitely many.\end{proof}

\subsection{Construction of pluriharmonic functions on $W$}\label{crphf}

Without loss of generality, we suppose that the value of $\alpha$ on $\widetilde{M}$ is $0$. Let $U$ be a cell. Due to the standing assumptions, $\p\alpha\wedge \bar\p\alpha\wedge\p\bar\p\alpha=0$ in $U$, and so, by Lemma \ref{LAURA2} there is a real analytic function $\mu_1:U\setminus\Crt(\alpha)\to\R$ such that $\p\bar\p\alpha=\mu_1\p\alpha\wedge\bar\p\alpha$ in $U\setminus\Crt(\alpha)$. Clearly, part e) of Corollary \ref{FUN} applies in the present situation so: ${\rm d}\mu_1\wedge{\rm d}\alpha=0$ in $U\setminus\Crt(\alpha)$. 

\begin{propos}\label{Ermia} Let $M^i$ be a face of the cell $U$ and consider, for a point $p\in M^i$, a neighbourhood $V_p$ containing $p$ such that $V_p\setminus M^i$ has exactly two connected components, denoted by $V_p^+$ and $V_p^-$. Then there exist an integer $k$, a pluriharmonic function $\widetilde{\chi}:V_p\to\R$ and a real-analytic, strictly increasing function $\lambda:\widetilde{\chi}(V_p)\to\R$ such that
$$\alpha(x)=(\lambda\circ \widetilde{\chi}(x))^{2k}$$
for every $x\in V_p$.\end{propos}
\begin{proof} Set $M_p=V_p\cap M$. By Proposition \ref{GLO2}, there exists $k\in \N$ such that $\alpha$ is flat of order $2k-1$ on $M_p\setminus T_p$ where $T_p$ is a real analytic subset of $M_p$, $\dim T_p\leq 2$. By Proposition \ref{GLO3}, the function
$$\beta(x)=\left\{\begin{array}{ccl}0&\textrm{if}&x\in M_p\setminus T_p\\\alpha(x)^{1/2k}&\textrm{if}&x\in V_p^+\\-\alpha(x)^{1/2k}&\textrm{if}& x\in V_p^-\end{array}\right.$$
is a real analytic function on $V_p\setminus T_p$, without critical points on $M_p\setminus T_p$; on $V_p\setminus M_p$, the critical points of $\beta$ are the same as those of $\alpha$, so $\dim(\Crt(\beta))\leq 2$.

Let $W_p=V_p\setminus T_p$; since $\beta=\pm \alpha^{1/2k}$ outside $M_p$, direct computation shows that $\p\beta\wedge\bar\p\beta\wedge \p\bar\p\beta=0$ and $\p\bar\p\beta\wedge\p\bar\p\beta=0$ on $V_p^+\cup V_p^-$ and hence, by analyticity, on $W_p$. By Lemma \ref{LAURA2}, there is a real analytic function $\mu_p:W_p\setminus \Crt(\beta)\to\R$ such that $\p\bar\p\beta=\mu_p\p\beta\wedge\bar\p\beta$.

We have, by direct computation,
$$\mu_p=\frac{1-2k}{\beta}+2k\beta^{2k-1}\mu_1$$
and since ${\rm d}\mu_1\wedge{\rm d\alpha}=0$ in $U$, ${\rm d}\mu_p\wedge{\rm d}\beta=0$ in $U\cap V_p$ hence in $W_p\setminus\Crt(\beta)$ by analyticity.

We are now in the position to apply Lemma \ref{LIS}: there exists a real analytic function $\theta:\beta(W_p)\to\R$, strictly increasing on the open interval $\beta(W_p)$ and such that $\chi:=\theta\circ\beta$ is pluriharmonic on $W_p$, with ${\rm d\chi}\wedge{\rm d\beta}=0$ in $W_p$ and so ${\rm d}\chi\wedge{\rm d\alpha}=0$ on $W_p$.

As $0\in \beta(W_p)$ and $\lim_{x\to T_p}\beta(x)=0$, we extend $\chi$ by continuity to $\widetilde{\chi}: V_p\to \R$, which is pluriharmonic on $W_p=V_p\setminus T_p$. By Lemma \ref{Silla}, $\widetilde{\chi}$ is indeed pluriharmonic on $V_p$.

We obtain the thesis by setting $\lambda=\theta^{-1}$.\end{proof}

We thus constructed a pluriharmonic function on $V_p$; up to now, such a function is not uniquely determined by our construction. In fact, we have two ``degrees of freedom'', the choice of $V_p^+$ and $V_p^-$ and the choice of the function $\theta$.
 
The latter issue is solved by Lemma \ref{LIS}, \emph{\rm{iii)}}, where it is proved that, once we choose such a function $\theta$, any other possible function is of the form $a\theta+b$; therefore, we can fix a particular $\theta_0$ by requiring that $\theta_0(0)=0$ and $\theta_0'(0)=1$ (as $0\in \beta(W_p)$). The pluriharmonic function $\widetilde{\chi}$ constructed on $V_p$ with this normalized function $\theta_0$ is denoted by $\chi_p$.

The following result is immediate.

\begin{corol} \label{Ippolita}Let $M^i$ be a face of the cell $U$ and consider, for a point $p\in M^i$, a neighbourhood $V_p$ containing $p$ such that $V_p\setminus M^i$ has exactly two connected components, denoted by $V_p^+$ and $V_p^-$. Then there exists a unique  pluriharmonic function $\chi_p:V_p\to\R$, positive in $V_p^+$, such that
$$\lim_{V_p^+\ni x\to M^i}\frac{(\chi_p(x))^{2k}}{\alpha(x)}=1$$
where $2k-1$ is the order of flatness of $\alpha$ along $M^i\cap V_p$.\end{corol}

On the other hand, the sign issue cannot be satisfactorily solved, particularly if we take into account that some face could be internal. Therefore, for every point $p\in M^i$, we associate a \emph{pair} of germs of pluriharmonic functions $\{(\chi_p, V_p), (-\chi_p, V_p)\}$, since an exchange between $V_p^+$ and $V_p^-$ will only produce a change in the sign of $\beta$ and hence of $\chi_p$.

This is indeed a well defined germ on every point of $M^i$: by Corollary \ref{Ippolita}, if we take another neighbourhood $V'_p$ of $p$, different from $V_p$, but still such that $V'_p\setminus M^i$ has two connected components, the pair of germs we obtain in the end of the construction will be the same, i.e. we will have $\{(\chi'_p, V'_p), (-\chi'_p, V_p)\}$  with $\chi_p=\pm\chi'_p$ on $V_p\cap V'_p$.

\begin{propos}\label{Titania} The function $\alpha$ has the same order of flatness along all the faces, including internal ones.\end{propos}
\begin{proof} Let $M^i$ be a connecting face between the cells $U^+$ and $U^-$. For any $p\in M^i$, we can take $V_p=W^i$, because $\dim_\R(W^i\cap\Crt(\beta))\leq 2$ and $\beta(W^i)=\beta(W^i\setminus\Crt(\beta))$ (as $\beta$ does not have absolute maximum or minimum points in $W^i$); therefore, we can apply Lemma \ref{LIS} to the domain $W^i$ and the function $\beta$.

By Corollary \ref{Ippolita}, we have a pluriharmonic function $\chi^i:=\chi_p$, canonical up to sign, such that there exists a strictly increasing function $\lambda:=\theta^{-1}$ so that
\begin{equation}
\label{Puck}
\alpha(x)=(\lambda\circ\chi^i)^{2k}(x)
\end{equation}
for $x\in W^i$. We also ask that $\chi^i>0$ on $U^+$ and $\chi^i<0$ on $U^-$.

We note that, as $\alpha\vert_{bW^i}=0$, also $\chi\vert_{bW^i}=0$. Now, let $M^j\subseteqq bW^i$ be a face; as $\chi^i>0$ on $U^+$ and $\chi^i<0$ on $U^-$, by Hopf Lemma \cite[Lemma 3.8]{ST}, we have that ${\rm d}\chi^i(x)\neq0$ for $x\in M^j$. From Equation \eqref{Puck}, it follows that the the order of flatness of $\alpha$ on $M^j$ is $2k-1$, for every $M^j$ adiacent to either $U^+$ or $U^-$.

Now, the set
$$H'=H\setminus\{\textrm{internal faces}\}$$
is open and connected, because every internal face is contained in the closure of a unique cell. Obviously every open set $W^i$ is connected and
$$H'=\!\!\!\!\!\!\bigcup_{M^i\textrm{ connecting}}\!\!\!\!\!\! W^i\;,$$
so, for every two cells $U^1$ and $U^2$ there exists a chain $W^{i_1},\ldots, W^{i_m}$, with $W^{i_h}\cap W^{i_{h+1}}\neq\emptyset$ for $h=1,\ldots, m-1$, connecting them inside $H'$. The order of flatness of $\alpha$ on the faces in ${\rm b}W^{i_1}$ (internal or connecting) will be the same as the order of flatness of $\alpha$ on the faces of $bW^{i_2}$ (internal or connecting), because there are connecting faces belonging to both, and so on, proving our thesis. \end{proof}

\begin{corol} 
Let $U$ be a cell and $M^i$ a connecting face between the cells $U$ and $U'$. For any face $M^j$ of $U$ and any $p\in M^j$, let $V_p$ be a neighbourhood of $p$ such that $V_p\setminus M^j$ has exactly two components. Then the pairs $\{(\chi^i, W^i), (-\chi^i, W^i)\}$ and $\{(\chi_p, V_p), (-\chi_p, V_p)\}$ induce the same germs in every $q\in V_p\cap W^i$.
\end{corol}
\begin{proof} 
This is a simple consequence of Corollary \ref{Ippolita}, once we know, from Proposition \ref{Titania}, that the order of flatness of $\alpha$ is the same along every face.
\end{proof}

Up to this point, we have given an open cover of $W$, constituted by the open sets $W^i$ corresponding to connecting faces and by the open sets $V_p$, for $p\in M^j$, as $M^j$ ranges among internal faces; on each of these open sets, we have constructed an \emph{unordered} pair of pluriharmonic functions, which differ just for the sign, i.e. of the form $\pm \chi$.

On the intersection of two such open sets, the pairs coincide, i.e. their restrictions to the intersection give the same pair of pluriharmonic functions; therefore, we may also describe what we have obtained so far as a subsheaf $\Chi$ of the sheaf of pluriharmonic functions on $W$, whose stalk at any point of $W$ is made exactly of two germs, which differ just for the sign.

\subsection{Extension to critical components of lower dimension}

We define the set
$$S:=X\setminus W=\bigcup_{i\in I}(\overline{M}^i\setminus M^i)\;,$$
consisting of the singular points of $\widetilde{M}$.
\begin{rem}\label{Oberon}
$S$ is of real dimension $2$ or lower. Given $p\in S$ and  a neighbourhood $V$ of it, if $\chi\in\Gamma(V\cap W, \Chi)$, then there is a unique continuous extension of $\chi$ to $\widetilde{\chi}:V\to \R$, obtained by setting $\widetilde{\chi}\vert_{S\cap V}\equiv 0$. By Lemma \ref{Silla}, $\widetilde{\chi}$ is then pluriharmonic on $V$.\end{rem}

Moreover, if $\Gamma(V\cap W,\Chi)\neq\emptyset$, then it necessarily has two elements, of opposite sign, both extending to pluriharmonic functions on $V$. Therefore, in order to extend our subsheaf $\Chi$ to the whole of $X$, we only need to show that for any point $p\in S$ there is a neighbourhood $V$ such that $\Gamma(W\cap V, \Chi)\neq\emptyset$.

The existence of sections depends \emph{prima facie} on the topological properties of $V\setminus S$ and of the combinatorial properties of the open cover of it given by the sets $W^i$ and $V_p$ defined above, where we have sections. The combinatorics of such cover is ultimately determined by the topology of $V\setminus \widetilde{M}$, which we propose to examine in the next pages.

To this aim, we employ the theory of stratified spaces, devoleped by Whitney, Mather, Thom and others in the 60s and 70s (see \cite{math} for the original paper, \cites{tro, pfl} for more detailed explanations); in what follows, we will just recall the main concepts and the results we need, avoiding many technical, although important, details that the interested reader can find in \cite{tro}. Precise references will be given in the next pages, as we state the relevant definitions and theorems.

\begin{defin}[Definition 1.1 in \cite{tro}]\label{strat}
Let $Z$ be a closed subset of a real analytic manifold $X$. A \emph{real analytic stratification of} $Z$ is a filtration by closed subsets
$Z=Z_d\supset Z_{d-1}\supset\cdots\supset Z_1\supset Z_0$ such that each difference $Z_i\setminus Z_{i-1}$ is a real analytic submanifold of $X$ and is of dimension $i$, or empty. Each connected component of $Z_i\setminus Z_{i-1}$ is called a {\it stratum}. Thus $Z$ is a disjoint union of strata.
\end{defin} 

Many patologies can be found among stratified spaces. To avoid them, at least to some extent, additional conditions are usually required to hold, namely the \emph{frontier condition} and \emph{Whitney's conditions $(a)$ and $(b)$}.

The frontier condition asks that, whenever $S$ and $T$ are two strata such that $S\cap\oli T\neq \emptyset$, we have $S\subseteqq\oli T$ (cfr \cite[Definition 1.2]{tro}). Whitney's conditions are more involved and impose some restrictions on the behaviour of tangent spaces when going from one stratum to another adiacent to it. For the precise statements see \cite[Definition 2.1]{tro}.

\begin{defin}\label{Wstrat} A locally finite (i.e whose strata are a locally finite family) stratification satisfying the frontier condition and Whitney's conditions $(a)$ and $(b)$ is called a \emph{Whitney stratification}.\end{defin}

In 1965, Whitney proved the following (see \cite[Theorem 2.1]{tro} and the references given therein).

\begin{teorema} Every real analytic variety admits a Whitney stratification whose strata are real analytic manifolds\end{teorema}

In particular, as a Whitney stratification is locally finite, a compact analytic variety admits a Whitney stratification with a finite number of strata. Whitney's theorem allows us to apply to real analytic varieties the following \emph{local topological triviality result} for stratified spaces (see \cite[Theorem 3.2]{tro}).

\begin{teorema}[Thom-Mather Tubular Neighbourhood Theorem]\label{tub} 
Let 
$$
Z=Z_d\supset Z_{d-1}\supset\cdots\supset Z_1\supset Z_0
$$
be a Whitney stratified subset of a real analytic manifold $X$. Then for every stratum $Y$ and each point $y_0\in Y$ there is a ``tubular'' neighborhood $G$ of $y_0$ in $X$, a stratified set {\rm (}a ``link''{\rm)} $L\subset S^{k-1}$ {\rm(}a $(k-1)$-dimensional sphere{\rm)} and a homeomorphism {\rm(}of stratified spaces{\rm)}
$$
h:\big(G,G\cap Z,G\cap Y\big)\longrightarrow\big(G\cap Y\big)\times\big(B^k,{\rm c}(L),O_k\big)
$$
where $k=\rcodim Y$ in $X$, $B^k$ is the open $k$-ball, ${\rm c}(L)$ is the cone on the link $L$ with vertex $O_k$, the center of $B^k$.
\bigskip
$L$ is stratified, 
$$
L=L_{d-k-1}\supset L_{d-k-2}\supset\cdots\supset L_1\supset L_0,
$$ 
induces a stratification of the cone
$$
{\rm c}(L)={\rm c}(L)_{d-k-1}\supset {\rm c}(L)_{d-k-2}\supset\cdots\supset {\rm c}(L)_1\supset {\rm c}(L)_0\supset \{O_k\}
$$
{\rm (}with $h(y_0)=O_k${\rm)}.
\end{teorema}

We also assume, without loss of generality, that $\oli{G}\cap Y=\oli{G}\cap Z_{d-k}$, so that, in particular, $\oli{G}$ intersects only one stratum of $Z_{d-k}$.

\medskip

We apply this theory to the analytic set $\widetilde{M}$; we have a stratification 
$$\widetilde{M}=Z_3\supset Z_2\supset Z_1\supset Z_0$$
where $S\subseteqq Z_2$. As $\widetilde{M}$ is compact, the number of strata is finite.

\begin{propos}\label{Ero}Let $Y\subseteqq Z_2\setminus Z_1$ be a $2$-dimensional stratum which is contained in $S$. For every point $p\in Y$ there exists a neighbourhood $V_p$ such that
$$\Gamma(V_p\cap W, \Chi)\neq\emptyset\;.$$
\end{propos}
\begin{proof} We apply Theorem \ref{tub} to $p\in Y$; we note that, in this case, $d=3$, $k=2$, so $L$ will be a (compact) link in $S^1$, i.e. a finite number $m$ of points. 

The cone over $L$ in $B^2$ is the union of $m$ radii $T_1,\ldots, T_m$, connecting the points of $L$ to the origin and dividing $B^2$ in $m$ connected open sectors $S_1,\ldots, S_m$, indexed so that $T_i, T_{i+1}\subset\oli{S}_i$, with $T_{m+1}=T_1$.

Let $G$ be the neighbourhood given by Theorem \ref{tub} and $h$ the homeomorphism of stratified spaces; we assume that $G\cap Y=G_0$ is a topological disc. 

The sets $N_j:=h^{-1}(G_0\times T_j)$, for $j=1,\ldots, m$, are open connected subsets of $Z_3\setminus Z_2$, i.e. of the faces in $\widetilde{M}$; likewise, $E_j:=h^{-1}(G_0\times S_j)$ are open connected subsets of $X\setminus Z_3=X\setminus\widetilde{M}$, i.e. of the cells.

We define, for $j=1, \ldots, m$, 
$$V_j=E_j\cup N_{j+1}\cup E_{j+1}\;,$$ 
where $N_{m+1}=N_1$ and $E_{m+1}=E_1$; we apply Corollary \ref{Ippolita} to each $V_j$, with $V_j^+=E_j$, obtaining a (unique) pluriharmonic function $\chi_j$.

Obviously, by the uniqueness of $\chi_j$, we have that 
$$\chi_j\vert_{E_{j+1}}=-\chi_{j+1}\vert_{E_{j+1}}\;,$$ 
for each $j=1,\ldots, m$.

Therefore, if $m$ is even, $\chi_1$ and $(-1)^{m-1}\chi_m=-\chi_m$ have the same sign on $E_1$, hence coincide. So we can glue together the functions $(-1)^{j-1}\chi_j$, for $j=1,\ldots, m$ into a pluriharmonic function 
$$\chi:\bigcup_{j=1}^m V_j\to\R$$
which is obviously a section of $\Chi$ on
$$\bigcup_{j=1}^mV_j=W\cap G\;.$$
Setting $V_p=G$ proves the thesis. 

In order to end the proof we have to show that $m$ is even. This fact is a consequence of the following Sullivan's Theorem on the local Euler characteristic of an analytic variety (cfr. \cite[Corollary 2]{sul}): a real analytic space is locally homeomorphic to the cone over a polyedron with even Euler characteristic.

Indeed, represent the neighborhood $G\cap Z$ of $y_0$ in $Z$ as an open cone over some polyhedron $K$ and compute $\uli\chi(K)$, the Euler characteristic of $K$. $G\cap Z=G_0\times{\rm c}(L_0)$, where $G_0$ is topologically a disk. It is clear that $G\cap Z$ is topologically an open cone with vertex at $(y_0,0)$ over its relative boundary in $Z$, denoted $K$. We consider $G_0$ as a single (open) triangle $(ABC)=\D$. Then $G_0\times{\rm c}(L_0)$ is the union of $m$ prisms $\D\times(O,P_i)\simeq\D\times(0,1)$ and of a single simplex $\D\{O\}$, which is the common boundary of all prisms (and so belongs to the open set $G_0\times{\rm c}(L_0)$). Compute now the Euler characteristic $\uli\chi(K)=c_0-c_1+c_2$, where $c_i$ is the number of $i$-dimensional simplexes in $K$.

$$
c_0=3m+3
$$
(vertices $(A,P_i)$, $(B,P_i)$, $(C,P_i)$, $i=0,1,\ldots,m-1$, and $(A,O)$, $(B,O)$, $(C,O)$).

We divide each of the $3$ side squares of the $m$-prisms into two simplexes in whatever way so
\begin{eqnarray*}
c_1=&3m& {\rm (from\, the\,\, tops\,\, of\,\,prisms)}\\
&+& 2\cdot3\cdot m\,\,{\rm (from\,\, the\,\, non\,\,horizontal\,\, side\,\, square)}\\
&+& 3\,\,{\rm (from\,\, the\,\,common\,\,base)}\\
&=& 9m+3.   
\end{eqnarray*}
Finally
\begin{eqnarray*}
c_2=&m& {\rm (triangles\,\,from\, the\,\, tops\,\, of\,\, prisms)}\\
&+& 3\cdot2\cdot m\,\,{\rm (non\,\,horizontal,\,\,from\,\, the\,\, side\,\, squares)}\\
&+& 0\,\,{\rm (none\,\,from\,\, the\,\,base)}\\
&=& 7m   
\end{eqnarray*}
so 
$$
\uli\chi(K)=3m+3-(9m+3)+7m=m.
$$
The proof of Proposition \ref{Ero} is now complete.
\end{proof}

\begin{propos}\label{Orsola}The sheaf $\Chi$ extends to the whole $X$.\end{propos}
\begin{proof} By Proposition \ref{Ero} and Remark \ref{Oberon}, we can extend the sheaf $\Chi$ to $X\setminus Z_1$.

Suppose now that $p\in Z_1\setminus Z_0$ and denote by $\tau$ the $1$-dimensional stratum with $p\in\tau$. By Theorem \ref{tub}, we have a neighbourhood $G$ and a homeomorphism $h$ such that $h(G\setminus \tau)=(G\cap \tau)\times (B^3\setminus \{O_3\})$.
Therefore, $G\cap (X\setminus Z_1)=G\setminus\tau$ is simply connected, which is enough to say that
$$\Gamma(G\cap (X\setminus Z_1), \Chi)\neq \emptyset\;.$$

Therefore, by Remark \ref{Oberon}, we can extend $\Chi$ on $X\setminus Z_0$. Let $p\in Z_0$ and consider a topological ball $B$ around $p$. Obviously, $B\setminus\{p\}$ is simply connected, hence we can repeat the previous argument and show that
$$\Gamma(G\cap(X\setminus Z_0), \Chi)\neq \emptyset\;.$$

Therefore, we obtain an extension of $\Chi$ to the whole $X$.\end{proof}

We are now in the position to prove Theorem \ref{maint}.

\medskip

\begin{proof}[Proof of Theorem \ref{maint}.] Let the extension of $\Chi$ to $X$, given by Proposition \ref{Orsola}, be denoted again by $\Chi$. Let $X^*$ be the total space of $\Chi$, together with the projection $\pi:X^*\to X$, which sends the germ $(h)_x\in \Chi_x$ to $x$, for every $x\in X$.

We give $X^*$ the natural complex structure, so that $\pi$ is a $2$-to-$1$ locally trivial holomorphic covering map, open and proper. We define the function
$$\chi^*:X^*\to\R$$
by $\chi^*((h)_x)=h(x)$, for $(h)_x\in \Chi_x$. It is easy to check that $\chi^*$ is a proper pluriharmonic function on $X^*$.

It may happen that $X^*$ is disconnected, then each of the two connected components is biholomorphic to $X$ and $\chi^*$ descends to a proper pluriharmonic function $\chi:X\to\R$.\end{proof}

\begin{rem} Together with Theorem \ref{SCOTTISH}, the previous result proves the second part of the Structure Theorem; from the existence of a global pluriharmonic function defining the foliation, we deduce that such a foliation is indeed holomorphic. Hence, both case ii and case iii of the Structure Theorem are examples of holomorphic foliations; from such observation, we can proceed as Brunella did in \cite{Bru}, noticing that there are examples of weakly complete surfaces which do not admit holomorphic foliations, but for which $\Sigma^1=X$, and concluding that such surfaces do not admit real analytic plurisubharmonic exhaustion functions.\end{rem}

\section{Appendix.}\label{app}

\noindent As a consequence of the Hopf Lemma in a weak form (cfr. \cite[Lemma 3.8]{ST}) we have the following
\begin{propos}\label{MARL} 
Let $W$ be a complex surface and $\beta:W\to \R$ a real analytic plurisubharmonic function. Then every $3$-dimensional connected component of ${\rm Crt}(\beta)\smallsetminus {\rm Sing}({\rm Crt}(\beta))$ consists of local minimum points of $\beta$ and so it is contained in $\widetilde{M}$. ${\rm Crt}(\beta)\smallsetminus{\widetilde{M}}$ is a real analytic subset of $W$ and ${\rm dim}_{\R}\,{\rm Crt}(\beta)\smallsetminus\widetilde{M}\le 2$.
\end{propos}

The following ``good covering lemma" is probably well known, so we just give a brief idea of the proof.
\begin{lemma}\label{gocov}
Let $W$ be a $2$-dimensional complex manifold and $C$ a smooth compact complex curve in $W$. Let $U$ be a neighborhood of $C$. Then there is a finite covering $\{V_j\}_j^m$ of $C$ by open subsets of $W$, such that
\begin{itemize}
\item[i)] $C\subset\bigcup\limits_{j=1}^mV_j\subset U$;
\item[ii)] every $V_j$ is simply connected;
\item[ijj)] $V_j\cap C$ is connected for $j=1, 2,\ldots,m$;
\item[iv)] whenever $V_j\cap B_k\neq\varnothing$, then $V_j\cap V_k\cap C\neq\varnothing$ and both sets $V_j\cap V_k$ and $V_j\cap V_k\cap C$ are connected, $1\le j,k\le m$. 
\end{itemize}
\end{lemma}
\begin{proof} Being smooth, $C$ possesses a tubular neighbourhood in $W$ which is homeomorphic to the product $C\times\Delta$, where $\Delta$ is the unit disc. It is now enough to cover $C$ with finitely many connected, simply connected open sets $\{\Omega_j\}_{j=1}^n$ such that their pairwise intersections are either empty or connected, e.g. convex balls for some Riemannian metric on $C$; the desired covering will be the preimage in the tubular neighbourhood of the products $\{\Omega_j\times\Delta\}_{j=1}^n$. \end{proof}
\subsection{Flatness}
Let $M$ be a $C^\infty$-smooth hypersurface in a differentiable manifold $W$ and $\beta$ a $C^\infty$-smooth function
defined in a neighbourhood of $M$. We recall that $\beta$ is said {\em flat} of order $N$
along $M$ in $X$, if
\begin{itemize}
\item[{\rm i)}] $\beta_{|M}=C$ (constant);\\
\item[{\rm ii)}] $\frac{\p^{k}\beta}{\p\nu^{k}}(p)= 0$, for $1\le k\le N$ and $\frac{\p^{N+1}\beta}{\p\nu^{N+1}}(p)\neq 0$, for every $p\in M$ ($\p/\p\nu$ normal derivative).
\end{itemize}
The normal derivative in question can be taken with respect to any Riemannian metric on $M$, and the property does not depend on the choice. If $x_1,\ldots,x_n$ are local coordinates on a neighbourhood $V$ of $p_0$ in $W$ chosen in such a way that $x_n\vert_{V\cap M}=0$, then $\beta$ is flat of order $N$ if and only if $\beta_{|V\cap M}$ is constant, $\frac{\p^{k}\beta}{\p x_n^{k}}(p)= 0$, for $1\le k\le N$ and $\frac{\p^{N+1}\beta}{\p x_n^{N+1}}(p)\neq 0$, $p\in V\cap M$.
\begin{propos}\label{GLO2}. Let $M$ be a real analytic hypersurface in a real analytic manifold $W$,and $\alpha$ a real analytic function defined in a neighbourhood of $M$ such that $\alpha$ has local minimum at each point of $M$. Then there is an integer $k\ge 1$ and a real analytic subset $T\subset M$, such that $\alpha$ is flat of order $2k-1$ along $M\smallsetminus T$.
\end{propos}
\begin{proof}
(Sketch). The fact being essentially local, we can assume without loss of generality that: $W$ is a
neighbourhood of zero in ${\R}^n$, $M=W\cap\{x_n=0\}$, $\alpha\ge c$ in $W$ and, since $\alpha$ has local minimum at each point of $M$, that $\alpha_{|M}=c$, $c\in\R$. Developping $\alpha$ in power series with respect to $x_n$ we obtain:
$$
\alpha(x',x_n)=c+\sum_{k=1}^{+\infty}\alpha_k(x')x_n^k,
$$
where $\alpha_s(x')$, real analytic function of $x'=(x_1,x_2,\ldots,x_k,\ldots, x_{n-1})$.
Let $N$ be the smallest $k$ such that $\alpha_k(x')$ is not identically $0$. Denote 
$$
T=\{(x',0)\in M:\alpha_k(x')=0\}.
$$
Then $T$ is a nowhere dense real-analytic subset of $M$. Observe that since
$\alpha-c\ge 0$ in $W$, $N$ cannot be odd and so $N=2k$, $k\ge 1$. It follows that
$\alpha$ is flat of order $2k-1$ along $M$.
\end{proof}
\begin{propos}\label{GLO3}. Let $W$ be a connected real analytic manifold and $M$ a closed real analytic 
hypersurface. Let $\beta:W\to [0,+\infty)$ be a non negative real analytic function on $W$, identically $0$ on $M$ and positive on $W\smallsetminus M$, flat of order $(2k-1)$ with $k\ge 1$, on $M\smallsetminus T$, where $T$ is a real analytic subset of $M$, nowhere dense in $M$. Assume that $M$ separates $W$ into two connected open sets $W^+$ and $W^-$. Let 
$$
\gamma(x)=\begin{cases}
0\>\>&if \>\>x\in M\smallsetminus T\\
\beta(x)^{1/2k}\>\>&if \>\>x\in W^+\\
-\beta(x)^{1/2k}\>\>&if \>\>x\in W^-.
\end{cases} 
$$
Then $\gamma:W\smallsetminus T\to\R$ is a real-analytic function without critical points on
$M\smallsetminus T$. Clearly $\beta=\gamma^{2k}$ on $W\smallsetminus T$.
\end{propos}
\begin{proof}
Since the function $\beta$ is well defined everywhere in $W\smallsetminus T$ and is,
obviously, analytic on $W\smallsetminus M$ (note the assumption $\beta>0$ on $W\smallsetminus M$), it
is enough to verify its properties at points of $M$. Since they are local,
assume, without loss of generality, the setting and notation of the proof of last proposition. That is 
$p\in M\smallsetminus T$, $M\subset\{x_n=0\}\subset\R^n$,
$$
\beta(x',x_n)=\sum_{s=2k}^{+\infty} x_n^s\alpha_s(x')=x_n^{2k}\rho(x',x_n),
$$
for a $(x',x_n)\in V$, a small neighbourhood, where
$$
\rho(x',x_n)=\sum_{l=0}^{+\infty} x_n^l\alpha_{2k+l}(x').
$$
It is clear that $\rho$ is analytic and $\rho> 0$ in $V$ ($V$ intersects $M$),
and that $(\beta_{|V})(x)=x_n\rho(x)^{1/2k}$. Since $\rho(x)>0$, $\rho(x)^{1/2k}$ is real
analytic. In addition, $\p\beta/\p x_n(x',0)=\rho(x',0)^{1/2k}>0$.
\end{proof}  
\subsection{Local maximum sets.}\label{lms}
\begin{lemma}\label{Silla}
{\rm(Removable singularities)}. Let $W$ be a complex surface and $A$ a real analytic subset of dimension $\le 2$. Let $\chi:W\to\R$ be a continuous function such that $\chi_{|W\setminus A}$ is pluriharmonic. 

Then $\chi$ is pluriharmonic on $W.$
\end{lemma}
\begin{proof}
The problem is local so we can assume that $W$ is an open subset of $\C^2$. We recall that in view of \cite{HP} if $E\subset W$ is a closed subset of $2$-Hausdorff measure $0$ every plurisubharmonic function on $W\setminus E$ extends to $W$ by a (unique) plurisubharmonic function. It follows that a continuous function on $W$ which is  pluriharmonic on $W\setminus E$ is pluriharmonic. This fact reduces the proof of Lemma \ref{Silla} to the case when $A$ is non-singular of pure dimension $2$: indeed, the singular set $S$ of $A$ is a semianalytic subset $2$-Hausdorff measure $0$ (\cite[Remark 3.1, p. 27]{GMT}).

Thus we assume that $A$ is a $2$-dimensional connected real analytic submanifold of $W$.

Let $C$ be the subset of the complex points of $A$: $C$ is a real analytic subset so either $C=A$ or it is of dimension $\le 1$. In the first case $A$ is a complex curve hence a pluripolar set so, by \cite[Theorem (5.24)]{dem} and what is preceding we obtain that $\chi$ is pluriharmonic. In the second one, away from a proper, closed real analytic subset $N$, $A$ is a totally real surface, hence in an open neighbourhood $V$ of $p\in A\setminus N$ there exist local holomorphic coordinates $z=x+iy, w=u+iv$ such that $A\cap V=\{y=v=0\}$.
In the latter, away from a proper, closed real analytic subset $N$, $A$ is a totally real surface so in an open neighbourhood $V\subset W$ of $p\in A\setminus N$ exist local holomorphic coordinates $z=x+iy, w=u+iv$ such that $z(p)=w(p)=0$, $A\cap V=\{y=v=0\}$. Let $\rho=y^2+v^2+v$. If $V$ is sufficiently small  the hypersurface $\{\rho=0\}$ is smooth, contains $A\cap V$ and the domain $V\cap\{\rho>0\}$ is simply connected with strongly Levi concave boundary. Then $\chi_{|V\cap\{\rho>0\}}$ as real part  of a holomorphic function extends through $A\cap V$. Thus $\chi$ is pluriharmonic on $V\setminus N$ whence is pluriharmonic in $V$ since $N$ has $2$-Hausdorff measure $0$
\end{proof}
\begin{lemma}\label{Mario}
{\rm(Reflection principle for pluriharmonic functions)}. Let $W$ be domain in a complex surface $X$ and $M\subset\oli W\setminus W$ a real analytic submanifold of $X$ of dimension $3$. Let $\chi:W\to\R$ be a non constant pluriharmonic function. Assume that
$$
\lim\limits_{\stackrel{x\to M}{x\in W}}\chi(x)=c\in\R.
$$
Then  
\begin{itemize}
\item[a)] there exist an open set $\widetilde{W}$ containing $W\cup M$ and a pluriharmonic function $\widetilde{\chi}:\widetilde{W}\to\R$ extending $\chi.$
\item[b)] If, in addition, $\chi(x)>c$ for $x\in W$, then ${\rm d}\widetilde{\chi}(x)\neq 0$ for $x\in M.$
\end{itemize} 
\end{lemma}
\begin{proof}
Let us prove that $M$ is Levi flat. Fix an open neighborhood $U\Subset X$ of a point $a\in M$ and any local defining function $\varphi$ for $M$ at $a$ (i.e. $U\cap W=\{p\in U:\varphi(p)<0\}$). Let $L_a$  and assume, by a contradiction, that the Levi form of $\varphi$ at $a$ is not vanishing when restricted to the complex tangent line to $M$ at $a$ i.e. $M\cap U$ is either strongly pseudoconvex or strongly pseudoconcave at $a$ as boundary of $W\cap U$. In the first case, near $a$, $W\cap U$ is filled by a family $\{D_\epsilon\}_\epsilon$ of analytic discs attached to $M$. Since $\chi_{|D_\epsilon}$ is harmonic in $D_\epsilon$ and constant on the boundary, it is constant near $a$ and consequently on whole of $W$, $\chi$ being analytic: contradiction. The proof in the pseudoncave case is similar in view of the fact that then $\chi$ locally extends through $M$ by a pluriharmonic function. This proves that $M$ is Levi flat. 

Since $M$ is real analytic and Levi flat, it is locally biholomorphic to a real hyperplane so we may assume tha $X$ is the open ball $B=\{\vert z\vert^2+ \vert w\vert^2<1\}$, $z=x+iy, w=u+iv$, $W=\{(z,w)\in B\ :\ v>0\}$. This reduces the proof of the lemma to the classical ``reflection principle'' for harmonic functions. In view of \cite[Theorem 4.12]{ABR}, $\chi$ extends to a harmonic function $\widetilde{\chi}: B\to \mathbb{R}$, which is real analytic and pluriharmonic on $W$, therefore pluriharmonic on $W$. Ths shows part a).

Part b) is an application of Hopf Lemma \cite[Lemma 3.8]{ST}. 
\end{proof}
\begin{lemma}\label{ATT}
Let $X$ be a complex surface and $Z$ a compact real analytic set of dimension $\le 2$. Suppose that $Z$ has the local maximum property. Then $Z$ is the union of finitely
many compact complex curves.
\end{lemma}
 \begin{proof}
The proof is based on the following classical result of Hartogs \cite{Har}. Let $U$, $V$ be domains in $\C^2$ and $Y\subset U\times V$ be a closed set, finitely sheeted over $U$, in the sense that 
\begin{itemize}
\item[i)] $\oli Y\cap{\rm b}V=\varnothing$ 
\item[ii)] $Y\cap(\{z\}\times\C)$ is finite, for every $z\in U$. 
\end{itemize}
Assume that $Y$ is pseudoconcave in $U\times V$. Then $Y$ is a complex analytic variety in $U\times V$.

By Hartogs result, it is enough to show that every point $p\in Z$ has a neighborhood $W=U\times V$, which, (with respect to some local holomorphic coordinate system at $p$) satisfies i) and ii), $X\setminus Z$ being pseudoconvex (see \cite[Theorem 2]{Sl0}).

If $p$ is any point in a $2$-dimensional connected component of the regular set of $Z$, and $L$ is a complex line through $p$, transversal to the $2$-dimensional (real) tangent plane $T_p(Z)$, then a small neighborhood of $p$ in $Z$ will do.

Consider now another point $p\in Z$, and a complex line $L$ through it, in a small coordinate neighborhood $W$ of $p$. In view of the above, $W\cap Z$ is the union of several open complex submanifolds and (a priori) several real analytic arcs (the latter is actually impossible). Each of them (submanifolds or arcs) either intersects $L$ at a discrete countable set, or is fully contained in $L$. The latter is possible only for finitely many $L$, and so there is $L$ through $p$ and a neighborhood $W_0$ of $p$ such that for every complex line $L^\ast$ parallel to $L$ and intersecting $W_0$, the set $Z\cap W_0$ is finite. It is to show that a smaller neighborhood $W_1$ satisfying i) and ii) can be selected, so that $W_1\cap Z$ is a complex variety. We omit further details.
\end{proof}


\begin{bibdiv}
\begin{biblist}

\bib{ABR}{book}{
   author={Axler, Sheldon},
   author={Bourdon, Paul},
   author={Ramey, Wade},
   title={Harmonic function theory},
   series={Graduate Texts in Mathematics},
   volume={137},
   edition={2},
   publisher={Springer-Verlag, New York},
   date={2001},
   pages={xii+259},
   isbn={0-387-95218-7},
   %review={\MR{1805196 (2001j:31001)}},
   doi={10.1007/978-1-4757-8137-3},
}
\bib{barr}{article}{
 author={Barrett, D.},
   title={Global convexity properties of some families of three-dimensional compact Levi-flat hypersurfaces},
   journal={Trans. Amer. Math. Soc., vol.},
   volume={332},
   number={2},
   date={1992},
   pages={459--474},
   issn={0002-9947},
%   review={\MR{0136766 (25 \#227)}},
 }  

   
   \bib{Bru}{article}{
   author={Brunella, Marco},
   title={On K\"ahler surfaces with semipositive Ricci curvature},
   journal={Riv. Math. Univ. Parma (N.S.)},
   volume={1},
   date={2010},
   number={2},
   pages={441--450},
   issn={0035-6298},
  % review={\MR{2789451 (2012c:32034)}},
}


\bib{CaNe}{book}{
   author={Camacho, C{\'e}sar},
   author={Lins Neto, Alcides},
   title={Geometric theory of foliations},
   note={Translated from the Portuguese by Sue E. Goodman},
   publisher={Birkh\"auser Boston, Inc., Boston, MA},
   date={1985},
   pages={vi+205},
   isbn={0-8176-3139-9},
   review={\MR{824240 (87a:57029)}},
   doi={10.1007/978-1-4612-5292-4},
}


\bib{Car}{article}{
   author={Cartan, Henri},
   title={Quotients of complex analytic spaces},
   conference={
      title={Contributions to function theory},
      address={Internat. Colloq. Function Theory, Bombay},
      date={1960},
   },
   book={
      publisher={Tata Institute of Fundamental Research, Bombay},
   },
   date={1960},
   pages={1--15},
 %  review={\MR{0139769 (25 \#3199)}},
}


\bib{doC}{book}{
  author={do Carmo, Manfredo Perdig{\~a}o},
   title={Riemannian geometry},
   series={Mathematics: Theory \& Applications},
   note={Translated from the second Portuguese edition by Francis Flaherty},
   publisher={Birkh\"auser Boston, Inc., Boston, MA},
   date={1992},
   pages={xiv+300},
   isbn={0-8176-3490-8},
 %  review={\MR{1138207 (92i:53001)}},
   doi={10.1007/978-1-4757-2201-7},
   }
   \bib{dem}{book}{
    author={Demailly, J.-P.},
     title={Complex, Analytic and Differential Geometry},
 publisher={Univ. Grenoble I, Institut Fourier},
     place={Grenoble},
      date={1997},
      }
      
\bib{DieOhs}{article}{
   author={Diederich, Klas},
   author={Ohsawa, Takeo},
   title={A Levi problem on two-dimensional complex manifolds},
   journal={Math. Ann.},
   volume={261},
   date={1982},
   number={2},
   pages={255--261},
   issn={0025-5831},
   %review={\MR{675738 (83m:32017)}},
   doi={10.1007/BF01456222},
}

   
   \bib{GMO}{article}{
   author={Gilligan, Bruce},
   author={Miebach, Christian},
   author={Oeljeklaus, Karl},
   title={Pseudoconvex domains spread over complex homogeneous manifolds},
   journal={Manuscripta Math.},
   volume={142},
   date={2013},
   number={1-2},
   pages={35--59},
   issn={0025-2611},
   %review={\MR{3080999}},
  doi={10.1007/s00229-012-0592-8},
}

\bib{GMT}{book}{
   author={Guaraldo, Francesco},
   author={Macr{\`{\i}}, Patrizia},
   author={Tancredi, Alessandro},
   title={Topics on real analytic spaces},
   series={Advanced Lectures in Mathematics},
   publisher={Friedr. Vieweg \& Sohn, Braunschweig},
   date={1986},
   pages={x+163},
   isbn={3-528-08963-6},
   %review={\MR{1013362 (90j:32001)}},
   doi={10.1007/978-3-322-84243-5},
}
   \bib{GU}{book}{
   author={Gunning, R. C.},
   title={Lectures on Riemann surfaces},
   series={Princeton Mathematical Notes},
   publisher={Princeton University Press, Princeton, N.J.},
   date={1966},
   pages={iv+254},
  % review={\MR{0207977 (34 \#7789)}},
}
  \bib{Har}{article}{
   author={Hartogs, Friedrich},
   title={\"Uber die singul\"aren Stellen einer analytischen Funktion besthenden Gebilde},
   journal={Acta Math.},
   volume={32},
   date={1909},
   number={},
   pages={57--79},
   %issn={0035-6298},
}
   \bib{HP}{article}{
   author={Harvey, R.},
   author={Polking, J.},
   title={Extending analytic objects},
   language={},
   journal={Comm. Pure Appl. Math.},
   volume={28},
   date={1975},
   pages={701--727},
   issn={},
   review={},
   }
   
   \bib{Hir}{article}{
  author={Hironaka, Heisuke},
   title={Desingularization of complex-analytic varieties},
   language={French},
   conference={
      title={Actes du Congr\`es International des Math\'ematiciens},
      address={Nice},
      date={1970},
   },
   book={
      publisher={Gauthier-Villars, Paris},
   },
   date={1971},
   pages={627--631},
   %review={\MR{0425170 (54 \#13127)}},
   }
      
   \bib{math}{article}{
   author={Mather, John},
   title={Notes on topological stability},
   journal={Bull. Amer. Math. Soc. (N.S.)},
   volume={49},
   date={2012},
   number={4},
   pages={475--506},
   issn={0273-0979},
  % review={\MR{2958928}},
   doi={10.1090/S0273-0979-2012-01383-6},
   }
   
    \bib{N}{book}{
   author={Narasimhan, Raghavan},
   title={Introduction to the theory of analytic spaces},
   series={Lecture Notes in Mathematics, No. 25},
   publisher={Springer-Verlag, Berlin-New York},
   date={1966},
   pages={iii+143},
 %  review={\MR{0217337 (36 \#428)}},
   }
   
   \bib{Nar}{article}{
   author={Narasimhan, Raghavan},
   title={The Levi problem in the theory of functions of several complex
   variables},
   conference={
      title={Proc. Internat. Congr. Mathematicians},
      address={Stockholm},
      date={1962},
   },
   book={
      publisher={Inst. Mittag-Leffler, Djursholm},
   },
   date={1963},
   pages={385--388},
   %review={\MR{0176096 (31 \#371)}},
}
   
   \bib{Ni}{article}{
   author={Nishino, Toshio},
   title={L'existence d'une fonction analytique sur une vari\'et\'e
   analytique complexe \`a deux dimensions},
   language={French},
   journal={Publ. Res. Inst. Math. Sci.},
   volume={18},
   date={1982},
   number={1},
   pages={387--419},
   issn={0034-5318},
   %review={\MR{660835 (83k:32015)}},
   doi={10.2977/prims/1195184029},
}
\bib{Ohs}{article}{
   author={Ohsawa, Takeo},
   title={Weakly $1$-complete manifold and Levi problem},
   journal={Publ. Res. Inst. Math. Sci.},
   volume={17},
   date={1981},
   number={1},
   pages={153--164},
   issn={0034-5318},
   %review={\MR{613939 (82j:32031)}},
   doi={10.2977/prims/1195186709},
}

\bib{pfl}{book}{
 author={Pflaum, Markus J.},
   title={Analytic and geometric study of stratified spaces},
   series={Lecture Notes in Mathematics},
   volume={1768},
   publisher={Springer-Verlag, Berlin},
   date={2001},
   pages={viii+230},
   isbn={3-540-42626-4},
   %review={\MR{1869601 (2002m:58007)}},
   }
   
   \bib{R}{article}{
author={Rea, C.},
   title={Levi-flat submanifolds and holomorphic extension of foliations},
   journal={Ann. Scuola Norm. Sup. Pisa (3)},
   volume={26},
   date={1972},
   pages={665--681},
%   review={\MR{0425158 (54 \#13115)}},
}
   \bib{SCH}{article}{
   author={Shcherbina, N. V.},
   title={On the polynomial hull of a graph},
   journal={Indiana Univ. Math. J.},
   volume={42},
   date={1993},
   number={2},
   pages={477--503},
   issn={0022-2518},
 %  review={\MR{1237056 (95e:32017)}},
   doi={10.1512/iumj.1993.42.42022},
} 

\bib{Sl0}{article}{
   author={S{\l}odkowski, Zbigniew},
   title={Analytic set-valued functions and spectra},
   journal={Math. Ann.},
   volume={256},
   date={1981},
   number={3},
   pages={363--386},
   issn={0025-5831},
   %review={\MR{626955 (83b:46070)}},
   doi={10.1007/BF01679703},
}
 \bib{Sl2}{article}{
 author={S{\l}odkowski, Zbigniew},
   title={Local maximum property and $q$-plurisubharmonic functions in
   uniform algebras},
   journal={J. Math. Anal. Appl.},
   volume={115},
   date={1986},
   number={1},
   pages={105--130},
   issn={0022-247X},
  % review={\MR{835588 (87j:32050)}},
   doi={10.1016/0022-247X(86)90027-2},
}
\bib{ST}{article}{
    author={S{\l}odkowski, Zibgniew},
   author={Tomassini, Giuseppe},
   title={Minimal kernels of weakly complete spaces},
   journal={J. Funct. Anal.},
   volume={210},
   date={2004},
   number={1},
   pages={125--147},
   issn={0022-1236},
  % review={\MR{2052116 (2005a:32037)}},
   doi={10.1016/S0022-1236(03)00182-4},
}
 
\bib{sul}{article}{
  author={Sullivan, D.},
   title={Combinatorial invariants of analytic spaces},
   conference={
      title={Proceedings of Liverpool Singularities---Symposium, I
      (1969/70)},
   },
   book={
      publisher={Springer, Berlin},
   },
   date={1971},
   pages={165--168},
 %  review={\MR{0278333 (43 \#4063)}},
}
	
\bib{tro}{article}{
  author={Trotman, David},
   title={Lectures on real stratification theory},
   conference={
      title={Singularity theory},
   },
   book={
      publisher={World Sci. Publ., Hackensack, NJ},
   },
   date={2007},
   pages={139--155},
%   review={\MR{2342910 (2008g:58009)}},
   doi={10.1142/9789812707499\_0004},
}


\bib{Ued}{article}{
   author={Ueda, Tetsuo},
   title={On the neighborhood of a compact complex curve with topologically
   trivial normal bundle},
   journal={J. Math. Kyoto Univ.},
   volume={22},
   date={1982/83},
   number={4},
   pages={583--607},
   issn={0023-608X},
  % review={\MR{685520 (84g:32043)}},
}
  \end{biblist}
\end{bibdiv}
\end{document}